%% file: feast_dpg.tex
\documentclass[12 pt]{amsart}

\input{definition}

\graphicspath{{figs/}}

\begin{document}

\title{Analysis of FEAST spectral approximations 
  using the DPG discretization}

\author[J. Gopalakrishnan]{Jay Gopalakrishnan}
\address{Portland State University, PO Box 751, Portland, OR 97207-0751, USA}
\email{gjay@pdx.edu}

\author[L.~Grubi\v{s}i\'{c}]{Luka Grubi\v{s}i\'{c}}
\address{University of Zagreb, Bijeni\v{c}ka 30, 10000 Zagreb, Croatia}
\email{luka.grubisic@math.hr}

\author[J. Ovall]{Jeffrey Ovall}
\address{Portland State University, PO Box 751, Portland, OR 97207-0751, USA}
\email{jovall@pdx.edu}

\author[B. Q. Parker]{Benjamin Q.~Parker}
\address{Portland State University, PO Box 751, Portland, OR 97207-0751, USA}
\email{bqp2@pdx.edu}

\begin{abstract}
  A filtered subspace iteration for computing a cluster of eigenvalues
  and its accompanying eigenspace, \revv{known as ``FEAST''}, has
  gained considerable attention in recent years. This work studies
  issues that arise when FEAST is applied to compute part of the
  spectrum of an unbounded partial differential operator.
  Specifically, \rev{when the resolvent
  of the partial differential operator is approximated by the
  discontinuous Petrov Galerkin (DPG) method,
  it is shown that there is no spectral pollution. The theory 
  also provides
  bounds on the discretization errors in the spectral
  approximations.} Numerical experiments
  for simple operators illustrate the theory and also indicate the value
  of the algorithm beyond the confines of the theoretical
  assumptions. The utility of the algorithm is illustrated by applying
  it to compute guided transverse core modes of a \revv{realistic
  optical} fiber.
\end{abstract}

\thanks{This work was partially supported by the AFOSR
  (through AFRL Cooperative Agreement \#18RDCOR018, under grant
  FA9451-18-2-0031),
  the Croatian Science Foundation grant HRZZ-9345, 
  bilateral Croatian-USA grant
  (administered jointly by Croatian-MZO and NSF) and 
  NSF grant DMS-1522471.
  The numerical studies
  were facilitated by the equipment acquired using NSF's Major
  Research Instrumentation grant DMS-1624776.}

\maketitle

\section{Introduction} \label{sec:introduction}

We study certain numerical approximations of the eigenspace associated
to a cluster of eigenvalues of a reaction-diffusion operator, namely
the unbounded operator $ \cA=-\Delta -\nu$ in $L^2(\om),$ whose domain
is $H^1_0(\Omega)$. Here $ \nu\in L^\infty(\Omega)$ and
$\Omega\subset\R^n$ is an open bounded set with Lipschitz boundary.
The eigenvalue cluster of interest is assumed to be contained inside a
finite contour $\Gamma$ in the complex plane $\C$.  The
\revv{computational}
technique is the FEAST algorithm~\cite{Poliz09}, which is now well
known as a subspace iteration, applied to an approximation of an
operator-valued contour integral over $\Gamma$. This technique
requires one to approximate the resolvent function
$z\mapsto R(z)=(z-\cA)^{-1}$ at a few points along the contour. The
specific focus of this paper is the discretization error in the final
spectral approximations when the discontinuous Petrov Galerkin (DPG)
method~\cite{DemkoGopal11} is used to approximate the resolvent.

Contour integral methods \cite{Beyn2012, Poliz09,GuttePolizTang15,
  SakurSugiu03}, such as FEAST, have been gaining popularity in
numerical linear algebra.  When used as an algorithm for matrix
eigenvalues, discretization errors are irrelevant, which explains the
dearth of studies on discretization errors within such algorithms.
However, in this paper, like in~\cite{GopalGrubiOvall18,
  HorniTowns19}, we are interested in the eigenvalues of a partial
differential operator on an infinite-dimensional space. In these
cases, practical computations can proceed only after discretizing the
resolvent of the partial differential operator by some numerical
strategy, such as the finite element method. We specifically focus on
the DPG method, a least-squares type of finite element method.

One of our motivations for considering the DPG discretization is that
it allows us to approximate $R(z)$ by solving a sparse Hermitian
positive definite system (even when $z - A$ is indefinite) using
efficient iterative solvers. Another practical reason is that it
offers a built-in (a posteriori) error estimator in the resolvent
approximation \revv{(see~\cite{CarstDemkoGopal14})},
thus immediately suggesting a straightforward
algorithmic avenue for eigenspace error control. The exploitation of
these advantages, including the design of preconditioners and adaptive
algorithms, are postponed to future work. The focus of this paper is
limited to obtaining {\it a priori} error bounds and convergence rates
for the computed eigenspace and accompanying Ritz values.

According to~\cite{GopalGrubiOvall18}, bounds on spectral errors can
be obtained from bounds on the approximation of the resolvent
$z \mapsto R(z)$. This function maps complex numbers to
bounded operators.  In~\cite{GopalGrubiOvall18}, certain finite-rank
computable approximations to $R(z)$, denoted by  $R_h(z)$, were considered
and certain abstract sufficient conditions were laid out for bounding
the resulting spectral errors. (Here $h$ represents some
discretization parameter like the mesh size.) This framework is
summarized in Section~\ref{AF}. Our approach to the analysis in this
paper is to verify the conditions of this abstract framework when
$R_h(z)$ is obtained using the DPG discretization.

One of our applications of interest is the fast and accurate
computation of the guided modes of optical fibers. In \revv{the} design and
optimization of new optical fibers, such as the emerging
microstructured fibers, one often needs to compute such modes many
hundreds of times for varying parameters. FEAST appears to offer a
well-suited method for this purpose. The Helmholtz operator arising
from the fiber eigenproblem is of the above-mentioned type (wherein
$\nu$ is related to the fiber's refractive index). In Section~\ref{fib}, we
will show the efficacy of the FEAST algorithm, combined with the DPG
resolvent discretization, by computing the modes of a commercially
marketed step-index fiber.

The outline of the paper is as follows.  In Section~\ref{AF} we
present the abstract theory from~\cite{GopalGrubiOvall18} pertaining
to FEAST iterations using discretized resolvents of unbounded
operators. In Section~\ref{DPG} we derive new estimates for
discretizations of a resolvent by the DPG method.
In Section~\ref{NS} we present benchmark results on problems with
well-known solutions which serve as a validation of the
method. Finally, in Section~\ref{fib} we apply the method to compute
the modes of a ytterbium-doped optical fiber.

\section{The abstract framework}\label{AF}

In this section, we summarize the abstract framework of
\cite{GopalGrubiOvall18} for analyzing spectral discretization errors
of the FEAST algorithm when applied to general selfadjoint
operators. Accordingly, in this section, $A$ is not restricted to the
reaction-diffusion operator mentioned in Section~\ref{sec:introduction}.
Here we let $A$ be a linear, closed, selfadjoint (possibly unbounded) operator
$\cA: \dom(\cA) \subseteq \cH \to\cH$ in a complex Hilbert space
$\cH$, whose real spectrum is denoted by $\Sigma(\cA)$. We are
interested in approximating a subset $\Lambda\subset \Sigma(\cA)$ that
consists of a finite collection of eigenvalues of finite multiplicity,
as well as its associated eigenspace $E$ (the span of all the
eigenvectors associated with elements of $\Lambda$).

The FEAST iteration uses a rational function
\begin{align}\label{ContourQuadrature}
r_N(\xi)=w_N + \sum_{k=0}^{N-1} w_k(z_k-\xi)^{-1}~.
\end{align}
Here the choices of  $w_k, z_k \in \CCC$ are typically motivated by quadrature approximations of the Dunford-Taylor integral
\begin{equation}
  \label{eq:DunfodTaylor}
  \revv{S = \frac{1}{2 \pi \ii}\oint_\Gamma R(z)\,dz,}
\end{equation}
where $R(z) = (z - A)^{-1}$ denotes the resolvent of $A$ at $z$.
Above, $\Gamma$ is a positively oriented, simple, closed contour
$\Gamma$ that encloses $\Lambda$ and excludes
$\Sigma(A)\setminus\Lambda$, so that $S$ is the
exact spectral projector onto $E$.
Define
\[
  S_N = r_N(\cA) = w_N + \sum_{k=0}^{N-1}w_k R(z_k).
\]
More details on examples of $r_N$ and their
properties can be found in~\cite{GuttePolizTang15, Poliz09}.

We are particularly interested in a further approximation of $S_N$
given by
\begin{align}\label{eq:SNh}
S^h_N&=w_N + \sum_{k=0}^{N-1}w_k R_h(z_k).
\end{align}
Here $R_h(z):\cH \to\cV_h$ is a finite-rank \revv{approximation of} the
resolvent $R(z)$, $\cV_h$ is a finite-dimensional subspace of a
\revv{complex Hilbert} space~$\cV$ embedded in $\cH$, and $h$ is a parameter
inversely related to $\dim(\cV_h)$ such as a mesh size parameter. Note
that there is no requirement that these resolvent approximations are
such that $S_N^h$ is selfadjoint. In fact, as we shall see later (see
Remark~\ref{rem:nonselfadj}), the $S_N^h$ generated by the DPG
approximation of the resolvent is not generally selfadjoint.

We consider a version of the FEAST iterations that use the above
approximations. Namely, starting with a subspace
$\Eh 0 \subseteq \cV_h$, compute
\begin{equation}
  \label{eq:1}
  \Eh\ell = S^h_N \Eh {\ell -1}, 
  \qquad 
  \text{ for } 
  \ell=1,2,\ldots. 
\end{equation}
If $A$ is a selfadjoint operator on a finite-dimensional space $\cH$
(such as the one given by a Hermitian matrix), then one may directly
use $S_N$ instead of $S_N^h$ in~\eqref{eq:1}.  This case is the well-studied FEAST iteration for Hermitian matrices, which  can approximate spectral clusters of $A$ that are
strictly separated from the remainder of the spectrum. In our abstract
framework for discretization error analysis, we place a similar 
separation assumption on the exact undiscretized spectral parts
$\Lambda $ and $\Sigma(A) \setminus \Lambda$.  Consider the following
strictly separated sets
$I_\gamma^y = \{ x \in \R: | x- y| \le \gamma\},$ and
$O_{\delta,\gamma}^y = \{ x \in \R: |x - y| \ge (1+\delta)\gamma\}$,
for some $y\in\RR$, $\delta>0$ and $\gamma>0$.  Using these sets and
the quantities
\begin{align}\label{ContractionFactor2}
W=\sum_{k=0}^{N}|w_k|,\quad\quad
\hat\kappa=
\frac{  \displaystyle\sup_{x \in O_{\delta,\gamma}^y}  |r_N(x)|}
{ \displaystyle \inf_{x \in I_\gamma^y}|r_N(x)|}.
\end{align}
we formulate \revv{a spectral separation assumption below.}

\begin{assumption}
  \label{asm:rN}
  There are $y\in\RR$,
  $\delta>0$ and $\gamma>0$ such that
  \begin{align}\label{SpectralGap}
    \Lambda\subset I_\gamma^y,
    \qquad 
    \Sigma(\cA)\setminus\Lambda \subset
    O_{\delta,\gamma}^y,
  \end{align}
  and that $r_N$ is a rational function of the
  form~\eqref{ContourQuadrature} with the following properties:
  \[
    z_k \notin\overline{\Sigma(\cA)},\quad
    W<\infty, \quad \revv{\text{and}} \quad
    \hat \kappa <1.
  \]
\end{assumption}

\begin{assumption}
  \label{asm:Vshort}
  The Hilbert space $\cV$ is such that $E \subseteq
  \cV \subseteq \cH$,  there is a $C_\cV>0 $ such that for all
  $u \in \cV$, $ \| u \|_\cH \le C_\cV \| u \|_\cV$, and $\cV$ is an
  invariant subspace of $R(z)$ \revv{for all $z$ in the resolvent set
    of $A$, i.e., $R(z) \cV \subseteq \cV$.}
  (We allow $\cV = \cH$, \revv{and further
  examples where $\cV \ne \cH$ can be found in \cite[\S2]{GopalGrubiOvall18}.)}
\end{assumption}
\begin{assumption}
  \label{asm:Rlim}
  The operators $R_h(z_k)$ and $R(z_k)$ are bounded in
  $\cV$ and satisfy
  \begin{equation}
    \label{eq:Rh-R}
    \lim_{h \to 0}\; \revv{\max_{k=0,\ldots, {N-1}}} \| R_h(z_k) - R(z_k) \|_{\cV} = 0.
  \end{equation}  
\end{assumption} 

\begin{assumption}
  \label{asm:Vh_in_dom(a)}
  Assume that $\cV_h$ is contained in $\dom(a)$, where
  $a(\cdot, \cdot)$ denotes the symmetric (possibly unbounded)
  sesquilinear form associated to the operator $A$ (as described in,
  say, \cite[\S10.2]{Schmu12} or \cite[\S5]{GopalGrubiOvall18}).
\end{assumption}

Various examples of situations where
one or more of these assumptions hold can be found in
\cite{GopalGrubiOvall18}.
Next, we proceed to describe the main consequences of these
assumptions of interest here.  Let
$\Lambda = \{ \lambda_1, \ldots, \lambda_m\},$ counting multiplicities,
so that $m = \dim(E)$.  By the strict separation of
Assumption~\ref{asm:rN}, we can find a curve $\Theta$ that encloses
$\mu_i = r_N(\lambda_i)$ and no other eigenvalues of $S_N$. By
Assumption~\ref{asm:Rlim}, $S_N^h$ converges to $S_N$ in norm, so for
sufficiently small $h$, the integral
\[
  P_h = \frac{1}{2 \pi \ii} \oint_\Theta (z - S_N^h)^{-1} \;dz
\]
is well defined and equals the spectral projector of $S_N^h$
associated with the contour $\Theta$.  Let $E_h$
denote the range of $P_h$. Now, let us turn to the
\revv{iteration~\eqref{eq:1}. We shall} tacitly assume
throughout this paper that $\Eh 0 \subseteq \cV_h$ is chosen so that
$\dim \Eh 0 = \dim (P_h \Eh 0) = m$. In practice, this is not
restrictive: we usually start with a larger than necessary $\Eh 0$ and
truncate it to dimension $m$ as the iteration progresses. 

In order to describe convergence of spaces, we need to measure the
distance between two linear subspaces $M$ and $L$ of $\cV$.  For this,
we use the standard notion of gap \cite{Kato95} defined by
\begin{equation}
  \label{eq:gapV}
\gap_\cV( M, L ) = \max \left[ \sup_{m \in U_M^\cV} \dist{\cV} ( m, L), \;
  \sup_{l \in U_L^\cV} \dist{\cV} ( l, M) \right],
\end{equation}
where $\dist{\cV}(x, S) = \inf_{s \in S} \| x - s\|_\cV$
and
$U_M^\cV$ denotes the unit ball
$ \{ w \in M: \; \| w\|_\cV = 1\}$ of $M$.
The set of approximations to $\Lambda$ is defined
by
\[
  \Lambda_h = \{ \lambda_h \in \RRR: \;\exists 0 \ne u_h \in E_h
  \text{ satisfying } a(u_h, v_h) = \lambda_h (u_h, v_h) \text{ for
    all } v_h \in E_h\}.
\]
In other words, $\Lambda_h$ is the set of Ritz values of the compression of $A$ on $E$.
The sets $\Lambda$ and $\Lambda_h$ are compared using \revv{the} Hausdorff
distance. We recall that the Hausdorff distance between two subsets
$\Upsilon_1, \Upsilon_2 \subset \RRR$ is defined by
\[
\revv{
\hdist( \Upsilon_1, \Upsilon_2) = 
\max 
\left[
\sup_{\mu_1 \in \Upsilon_1} \hdist( \mu_1, \Upsilon_2), 
\sup_{\mu_2 \in \Upsilon_2} \hdist( \mu_2, \Upsilon_1)
\right],
}
\]
where $\hdist(\mu, \Upsilon) = \inf_{\nu \in \Upsilon} | \mu - \nu|$ for any 
$\Upsilon \subset \R.$  Finally, let $C_E$
denote any finite positive constant
satisfying
$a(e_1, e_2) \le C_E \| e_1 \|_\cH \| e_2 \|_{\cH}$ for all $e_1, e_2
\in E$. 
We are now ready to state collectively the following results proved
in~\cite{GopalGrubiOvall18}. 

\begin{theorem}
  \label{thm:abstract}
  Suppose Assumptions~\ref{asm:rN}--\ref{asm:Rlim} hold.
  Then there
  are constants $C_N, h_0>0$ such that, for all $h < h_0$, 
  \begin{gather}
    \label{eq:Eh_ell_to_Eh}
    \lim_{\ell \to \infty} \gap_\cV( \Eh \ell, E_h)  = 0, 
    \\
    \label{eq:E_to_Eh}
    \lim_{h\to 0} \gap_\cV (E, E_h)   = 0,
    \\
    \label{eq:gap_E_Eh}
    \gap_\cV( E, E_h) \le C_N W \max_{k=0,\ldots, {N-1}} 
    \left\| \big[  R(z_k) - R_h(z_k) \big]
      \raisebox{-0.1em}{\ensuremath{\big|_E}} \right\|_\cV.
  \end{gather}
  If, in addition, Assumption~\ref{asm:Vh_in_dom(a)} holds and
  $\| u \|_\cV = \| |A|^{1/2} u \|_\cH$, then there are $C_1, h_1>0$
  such that for all $h < h_1,$
  \begin{gather}
    \label{eq:dist_L_Lh}
    \hdist( \Lambda, \Lambda_h) \le (\lmax)^2 \gap_\cV(E, E_h)^2
    + 
    C_1 C_E \, \gap_\cH(E, E_h)^2,
  \end{gather}
  where $\lmax = \sup_{e_h \in E_h} \| |A|^{1/2} e_h\|_{\cH} / \| e_h
  \|_\cH$ satisfies 
  $
  (\lmax)^2\le \left[ 1-\gap_\cV(E, E_h) \right]^{-2} C_E.
  $
\end{theorem}

\section{Application to a DPG discretization}\label{DPG}

In this section, we apply the abstract framework of the previous
section to obtain convergence rates for eigenvalues and eigenspaces
when the DPG discretization is used to approximate
the resolvent of a model operator.

\subsection{The Dirichlet operator}

Throughout this section, we  set $\cH, \cV,$ and $A$ by 
\begin{equation}
  \label{eq:VHA_Dirichlet}
{  \cH = L^2(\om), \quad
  \cA = -\Delta, \quad
  \dom(\cA) = \{ \psi\in H_0^1(\om):~\Delta\psi\in L^2(\om)\} ,\quad
  \cV = H_0^1(\om),}
\end{equation}
where $\om$ $\subset \RR^n$ ($n\ge 2$) is a bounded polyhedral domain
with Lipschitz boundary.  We shall use standard notations for
norms ($ \| \cdot \|_X$) and seminorms ($|\cdot|_X$) on Sobolev
spaces~($X$).  It is easy to see \cite{GopalGrubiOvall18} that
Assumption~\ref{asm:Vshort} holds with these settings.  Note that the
operator $A$ in~\eqref{eq:VHA_Dirichlet} is the operator associated to the form
\[
a(u,v) = \int_\om \grad u \cdot \grad \overline{v} \; dx, \quad 
u, v \in \dom(a) = \cV = H_0^1(\om)
\]
and that the norm $\| u \|_\cV,$ due to the \Poincare\ inequality, is
equivalent to $\| |A|^{1/2} u \|_\cH $ $ = \| A^{1/2} u \|_\cH$ $= \| \grad u \|_{L^2(\om)} $
$= |u|_{H^1(\om)}$.

The \revv{solution of} the operator equation $(z - \cA ) u = v$ yields the
application of the resolvent $u = R(z) v$.
\revv{The weak form of this equation}
may be stated as the problem of finding $u \in H_0^1(\om)$
satisfying
\begin{equation}
	\label{eq:Rzv-weak}
	b(u,w) = (v,w)_\cH
\qquad \text{ for all } w\in H_0^1(\om), 
\end{equation}
where
\[
\revv{b(w_1,w_2) = z(w_1,w_2)_\cH - a(w_1,w_2)}
\]
for any $w_1, w_2 \in H_0^1(\om)$.
As a first step in the analysis, we obtain an inf-sup estimate and a
continuity estimate for $b$. In the ensuing lemmas $z$ is tacitly
assumed to be in the resolvent set of $A$.

\begin{lemma}
	\label{lem:dirichlet}
	For all $v
	\in H_0^1(\om)$,
\begin{align*}
\revv{
\sup_{y \in H_0^1(\om)} 
  \frac{| b(v,y) |}{ \quad| y |_{H^1(\om)}}\geq
  \beta(z)^{-1} | v |_{H^1(\om)},
}
\end{align*}
where $\beta(z)=\sup\{|\lambda|/|\lambda-z|:\;\lambda\in\Sigma(A)\}$.
\end{lemma}
\begin{proof}
  Let $v \in H_0^1(\om)$ be non-zero, and let
  $w = \overline{z}R(\overline{z})v$. Then
  \[
    b(s,w)=z(s,v)_\cH, \qquad \text{ for all }
    s\in H_0^1(\om).
  \]
  Choosing $s=v$, it follows immediately that
  \begin{equation}
    \label{eq:2}
    b(v, v-w) = b(v,v) - z \| v \|^2_{L^2(\om)} = -|v |_{H^1(\om)}^2.    
  \end{equation}
  Moreover,
  $v-w=(I-\overline{z}R(\overline{z}))v=-A\, R(\overline{z})
  v$. Recall that the identity $\| A R(z) \|_\cH = \beta(z)$
  holds~\cite[p.~273, Equation (3.17)]{Kato95} for any $z$ in the
  resolvent set of $A$. Since $|s|_{H^1(\om)} = \| A^{1/2}s\|_\cH$ for all
  $s \in H_0^1(\om)=\dom(a)=\dom(A^{1/2})$, and since $A^{1/2}$
  commutes with $A R(z)$, we conclude that
  \begin{equation}
    \label{eq:3}
    |v - w|_{H^1(\om)}
    = |A R(\overline{z}) v|_{H^1(\om)}
    = \| A R(\overline{z}) A^{1/2} v \|_\cH \le \beta(\overline{z}) \| A^{1/2} v
    \|_\cH =
     \beta({z}) |v|_{H^1(\om)}~,
  \end{equation}
  where $\beta(\overline{z})=\beta(z)$ because the spectrum is real.
  It follows from~\eqref{eq:2} and~\eqref{eq:3} that 
	\begin{align*}
		\sup_{y \in H_0^1(\om)} 
		\frac{| b(v, y) |}{ \quad| y |_{H^1(\om)}} 
		& \ge 
		\frac{| b(v,v-w) |}{ \quad| v-w |_{H^1(\om)}} 
                   \ge 
		\frac{ |v|_{H^1(\om)}^2 }{ \beta(z)  | v |_{H^1(\om)}}~,
	\end{align*}
	as claimed.
\end{proof}

\subsection{The DPG resolvent discretization}  \label{ssec:dpg}

We now assume that $\om$ is partitioned by a conforming simplicial
finite element mesh $\oh$. As is usual in finite element theory, while
the mesh need not be regular, the shape regularity of the mesh is
reflected in the estimates. 

To describe the DPG discretization of $z-A$, we begin by introducing
the nonstandard variational formulation on which it is based. We will
be brief as the method is described in detail in previous
works~\cite{DemkoGopal11,DemkoGopal13a}.  
Define
\[
H^1(\oh) = \prod_{K \in \oh} H^1(K),
\qquad
Q = H(\div, \om) / \prod_{K \in \oh} H_0(\div, K), 
\]
normed respectively by
\[
  \| v \|_{H^1(\oh)} = \left(
    \sum_{K \in \oh} \| v \|_{H^1(K)}^2
  \right)^{1/2},
  \qquad
  \| q \|_Q
  = \inf\left\{ \| q -  q_0\|_{H(\div, \om)}:
    \;
    q_0 \in \displaystyle{ \prod_{K \in \oh}}    H_0(\div, K)\right\}.
\]
On every mesh element $K$ in $\oh$, the trace $q\cdot n|_{\partial K}$
is in $H^{-1/2}(\partial K)$ for any $q $ in $H(\div,K)$.
\revv{Above, $H_0(\div, K) = \{ q \in H(\div, K) : \left. q
    \cdot n \right|_{\partial K} = 0 \}$.}
We denote by
$\ip{q\cdot n, v}_{\partial K}$ the action of this functional on the
trace $v|_{{\partial} K}$ for any $v$ in $H^1(K)$.  Next, for any
$u \in H_0^1(\om)$, $q \in Q$ and $v \in H^1(\oh)$, set
\[
b_h( (u,q), v) 
= \sum_{K \in \oh} 
\left[
  \ip{ q \cdot n, \bar v}_{\partial K}
  + \int_K \left( z u \bar v - \grad u \cdot \grad \bar v \right)\, \rev{dx}
\right].
\]
This sesquilinear form gives rise to a well-posed Petrov-Galerkin
formulation, as will be clear from the discussion below.

For the DPG discretization, we use the following finite element
subspaces. Let $L_h$ denote the Lagrange finite element subspace of
$H_0^1(\om)$ consisting of continuous functions, which when restricted
to any $K$ in $\oh,$ are in $P_{p}(K)$ for some $p\ge 1$.  Here and
throughout, $P_\ell(K)$ denotes the set of polynomials of total degree
at most~$\ell$ restricted to $K$. Note that when applying the earlier abstract
framework to the DPG discretization, in addition
to~\eqref{eq:VHA_Dirichlet}, we also set
\begin{equation}
  \label{eq:Vh_DPG}
  \cV_h = L_h.
\end{equation}
Let $\RT_h \subset H(\div, \om)$
denote the well-known Raviart-Thomas finite element subspace
consisting of functions whose restriction to any $K \in \oh$ is a
polynomial in \revv{$P_{p-1}(K)^n + x P_{p-1}(K)$}, where $x$ is the
coordinate vector. Then we set
$Q_h= \{ q_h \in Q: \; q_h |_K \in \revv{P_{p-1}(K)^n + x P_{p-1}(K)} + H_0(\div, K)\}$.
Finally, let $Y_h \subset H^1(\oh)$ consist of functions \revv{which,
when restricted to any $K\in \oh$,} lie in $P_{p+n+1}(K)$.

We now define the approximation of the resolvent action $u = R(z)f$ by
the DPG method, denoted by $u_h = R_h(z) f$, for any $f \in L^2(\om)$.
The function $u_h$ is in $L_h$. Together with $\veps_h \in Y_h$ and
$q_h \in Q_h$, it satisfies
\begin{subequations}
	\label{eq:dpg-eq}
	\begin{align}
          ( \veps_h,  \eta_h)_{H^1(\oh)} \, 
		+ b_h (( u_h, q_h), \eta_h) 
		& = \int_\om f \, \bar\eta_h \, \rev{dx},
		&& \text{ for all } \eta_h \in Y_h,
		\\
		b_h( ( w_h, r_h), \veps_h) & = 0, 
		&& \text{ for all } w_h \in L_h, \; r_h \in Q_h.
	\end{align}
 \end{subequations}
 where
 \[
   \revv{
   ( \veps_h, \eta_h)_{H^1(\oh)}  =
   \sum_{K \in \oh} \int_K
   ( \veps_h \bar{\eta}_h + \grad \veps_h \cdot \grad \bar{\eta}_h )
   \,\rev{dx}
   .}
 \]
The distance between $u$ and $u_h$ is bounded in the next
result. There and in similar results in
the remainder of this section, we tacitly understand  $z$ to vary in
some bounded  subset $D$ of the resolvent set of $A$ in the complex
plane (containing the contour $\Gamma$)
and write $t_1 \lesssim t_2$ whenever  there is a positive 
constant $C$
satisfying 
$t_1 \le C t_2$
and $C$ is independent of
\[
  h = \max_{K \in \oh}  \diam( K)
\]
but dependent on the diameter of $D$ and the shape regularity of the
mesh $\oh$. The deterioration of the estimates as $z$ gets close to the
spectrum of $A$ is identified using $\beta(z)$ \revv{of Lemma~\ref{lem:dirichlet}}.

\begin{lemma}
	\label{lem:dpg-resolvent}
        For all $f \in L^2(\om)$,
	\[
	\revv{
	\| R(z) f - R_h(z) f \|_\cV \lesssim
	\beta(z)
	\left[
	\inf_{w_h \in L_h} \| u - w_h \|_{H^1(\om)}
	+ 
	\inf_{q_h \in \RT_h} \| q - q_h \|_{H(\div,\om)}
	\right],
	}
	\]
	where $u= R(z) f$ and $q = \grad u$.
\end{lemma}
\begin{proof}
  The proof proceeds by verifying the sufficient conditions for
  convergence of DPG methods known in the existing literature.  The
  result of~\cite[Theorem~2.1]{GopalQiu14} immediately gives the
  stated result, provided we verify its three conditions, reproduced
  below in a form convenient for us.
  The first two conditions there, taken together, is equivalent
  to the bijectivity of the operator generated by $b_h(\cdot,\cdot)$.
  Hence we shall state them in the following alternate form (dual to
  the form stated in~\cite{GopalQiu14}). The first 
  is the uniqueness condition
  \begin{subequations}    \label{eq:A}
    \begin{gather}
      \label{eq:A1}
      \{ \eta \in H^1(\oh) : \; b_h((w,r), \eta) = 0,\;
      \text{ for all } (w,r) \in H_0^1(\om) \times Q \} = \{ 0 \}.
    \end{gather}
    The second condition is that there are  $C_1, C_2 >0$ such that
    \begin{equation}
      \label{eq:A2}
      C_1 \big[ |w|_{H^1(\om)}^2 + \| r \|_Q^2\big]^{1/2}
      \le
      \sup_{ \eta \in H^1(\oh)}
      \frac{ |b_h ( (w,r), \eta) | }{ \| \eta\|_{H^1(\oh)}
      }
      \le C_2 \big[ |w|_{H^1(\om)}^2 + \| r \|_Q^2\big]^{1/2}
    \end{equation}
    for all $w \in H_0^1(\om)$ and $r \in Q.$
    Finally, the third condition is the existence of a bounded linear operator
    $\varPi_h : H^1(\oh) \to Y_h$ such that 
    \begin{equation}
      \label{eq:A3}
      b_h( (w_h, r_h), \eta - \varPi_h \eta) =0.
    \end{equation}
  \end{subequations} 
  Once these conditions are verified,
  \cite[Theorem~2.1]{GopalQiu14} implies 
  \begin{equation}
    \label{eq:4}
    |u - u_h |_{H^1(\om)}
    \revv{\le}
    \frac{C_2 \| \varPi\| }{C_1}
    \left[
      \inf_{w_h \in L_h} | u - w_h |_{H^1(\om)}
      + 
      \inf_{q_h \in \RT_h} \| q - q_h \|_{H(\div,\om)}
    \right]
  \end{equation}
  with $u=R(z)f$ and $u_h = R_h(z) f$.

  It is possible to verify conditions~\eqref{eq:A1} and~\eqref{eq:A2}
  on $b_h(\cdot,\cdot)$ using the properties of $b(\cdot,\cdot)$.
  First note that~\cite[Theorem~2.3]{CarstDemkoGopal16} shows that
  \[
    \sup_{v \in H^1(\oh)} \frac{|\sum_{K \in \oh}\ip{r \cdot n, v}_{\partial K}|}{\|  v \|_{H^1(\oh)}} =
    \| r \|_Q.
  \]
  This, together with~\cite[Theorem~3.3]{CarstDemkoGopal16}  implies
  that the 
  inf-sup condition for $b$ that we proved in
  Lemma~\ref{lem:dirichlet} implies an inf-sup condition for~$b_h$,
  namely the lower inequality of~\eqref{eq:A2} holds with
  \[
    \revv{
    \frac{1}{C_1^2}
    = \beta(z)^2 +
    \left[ \beta(z)(1+|z|) +1 \right]^2 .
    }
  \]
  \revv{By combining this with} the continuity estimate of $b_h$ with $C_2 = 1+|z|$,
  we obtain that $C_2/C_1$ is $O(\beta(z))$. Finally,
  Condition~\eqref{eq:A3} follows from the Fortin operator constructed
  in~\cite[Lemma ~3.2]{GopalQiu14} whose norm is a constant bounded
  independently of~$z$. Hence the lemma follows from~\eqref{eq:4}.
\end{proof}

\begin{remark}  \label{rem:Pih}
  Note that the degree of functions in $Y_h$ was chosen to be $p+n+1$
  in order to satisfy the moment condition
  \[
    \int_K ( \eta - \varPi_h \eta ) w_p \, \rev{dx}= 0 
  \]
  for all $w_p \in P_p(K)$ and $\eta \in H^1(K)$ on all mesh simplices
  $K$ (see~\cite{GopalQiu14}). This moment condition was used while
  verifying~\eqref{eq:A3}.  Other recent ideas, such as those in
  \cite{BoumaGopalHarb14, CarstHellw18}, may be used to reduce $Y_h$
  without reducing convergence rates, and thus improve
  Lemma~\ref{lem:dpg-resolvent} for specific meshes and
  degrees.
\end{remark}

\begin{remark} \label{rem:nonselfadj}
  The DPG approximation of $u = R(z) f$, given by $u_h = R_h(z) f$,
  satisfies~\eqref{eq:dpg-eq}. We may rewrite~\eqref{eq:dpg-eq} using
  $x_h = (u_h, q_h)$,
  \begin{align*}
    M_h \veps_h + B_h x_h & = f_h\revv{,}\\
    \revv{B_h^* \veps_h} & \revv{= 0.}
  \end{align*}
  We omit the obvious definitions of operators
  $B_h: L_h \times Q_h \to Y_h,$ $M_h : Y_h \to Y_h$, and that of
  $f_h$ (an appropriate projection of $f$). Eliminating $\veps_h$, we
  find that $u_h = R_h(z) f$ is a component of
  \revv{$x_h = (B_h^* M_h^{-1} B_h)^{-1} B_h^* M_h^{-1}f_h$}. Thus, the operator $R_h(z)$
  produced by the DPG discretization need not be selfadjoint even when
  $z$ is on the real line. For the same reason, the filtered operator
  $S_N^h$ produced by the DPG discretization is {\em not generally
    selfadjoint} even when $\{z_k: k=0, \ldots, N-1\}$ has symmetry
  about the real line. \rev{Note that selfadjointness of $S_N^h$ is
    not needed in Theorem~\ref{thm:abstract} to conclude the
    convergence of the eigenvalue cluster at double the convergence
    rate of eigenspace.}
\end{remark}

\subsection{FEAST iterations with the DPG discretization}\label{sec:FEASTDPG}

To approximate \revv{$E \subseteq \cV$}, we apply the filtered subspace
iteration~\eqref{eq:1}. In this subsection, we complete the analysis
of approximation of $E$ by $E_h$ and the accompanying eigenvalue
approximation errors.  The analysis is an application of the abstract
results in Theorem~\ref{thm:abstract}. To verify the conditions of
this theorem, we need some elliptic regularity. This is formalized in
the next regularity assumption.

\begin{assumption}
  \label{asm:reg}
  Suppose there are positive constants $\Creg$ and $s$ such that the solution
  $u^f \in \cV$ of the Dirichlet problem $-\Delta u^f = f$ admits the
  regularity estimate
\begin{equation}
	\label{eq:reg}
	\| u^f \|_{H^{1+s} (\om)} \le \Creg \| f \|_\cH
        \quad\text{ for any } f \in \cV.
\end{equation}
\revv{Also suppose} that 
\begin{equation}
  \label{eq:reg-eig}
  \| u^f \|_{H^{1+{s_E}} (\om)} \le \Creg \| f \|_\cH
  \quad\text{ for any } f \in E.
\end{equation}
\revv{(Since $E \subseteq \cV,$ \eqref{eq:reg}
  implies~\eqref{eq:reg-eig}
  with $s$ in
  place of $s_E$, but in many cases~\eqref{eq:reg-eig} holds with
  $s_E$  larger than $s$. This is  the reason for additionally assuming~\eqref{eq:reg-eig}.)}
\end{assumption}

Its well known that if $\om$ is convex, Assumption~\ref{asm:reg} holds
with $s=1$ in~\eqref{eq:reg}. If $\om\subset\RR^2$ is non-convex, with
its largest interior angle at a corner being $\pi/\alpha$ for some
$1/2<\alpha<1$, Assumption~\ref{asm:reg} holds with any positive
$s<\alpha$. These results can be found
in~\cite{Grisv85}, for example. 

\begin{lemma}
  \label{lem:rates}
  Suppose Assumption~\ref{asm:reg} holds. Then, 
  \begin{align}
    \label{eq:R-R_h_on_V}
    \| R(z) f - R_h(z) f \|_\cV
    & \lesssim {\beta(z)^2} h^{\min(p, s, 1)}    \| f \|_\cV,
    && \text{ for all } f \in \cV,
    \\
    \label{eq:R-R_h_on_E}
    \| R(z) f - R_h(z) f \|_\cV
    & \lesssim {\beta(z)^2} h^{\min(p, s_E)}    \| f \|_\cV,
    && \text{ for all } f \in E.
  \end{align}
\end{lemma}
\begin{proof}
  By Lemma~\ref{lem:dpg-resolvent}, the distance between $u = R(z) f$
  and $u_h = R_h(z) f$ can be bounded using standard finite element
  approximation estimates for the Lagrange and Raviart-Thomas spaces,
  to get 
  \begin{equation}
    \label{eq:approx1}
    \| u - u_h \|_{H^1(\om)}
    \lesssim
    \beta(z)
    \bigg[
    h^r |u|_{H^{1+r}(\om)} + h^r |q|_{H^r(\om)} + h^r |\div \,q |_{H^r(\om)}
    \bigg],
    \qquad \text{ for } r \le p,
  \end{equation}
  where $q=\grad u$.  Note that since $u$ satisfies
  $b(u, v) = (f, v)_\cH$ for all $v \in H_0^1(\om)$, by
  Lemma~\ref{lem:dirichlet},
  \begin{equation}
    \label{eq:5}
    \beta(z)^{-1} |u|_{H^1(\om)}
    \le \sup_{y \in H_0^1(\om)}
    \frac{| b( u, y)|}{ |y|_{H_0^1(\om)}}
    =
     \sup_{y \in H_0^1(\om)}
     \frac{| (f, y)_\cH|}{ |y|_{H_0^1(\om)}} = \| f \|_{H^{-1}(\om)}.    
  \end{equation}
  which implies, by the \Poincare\ inequality, 
  \begin{equation}
    \label{eq:6}
    \| u \|_\cH  \lesssim
    |u |_\cV \lesssim
    \beta(z) \| f \|_{H^{-1}(\om)} \lesssim
    \beta(z) \| f \|_\cH.
  \end{equation}
  Applying elliptic regularity  to $ \Delta u = f - zu$, for all $r
  \le s$ and $r\le 1,$
  \begin{align}
    \nonumber 
    | u |_{H^{1+r}(\om)}
    &  \le \Creg ( \| f \|_\cH + |z| \| u \|_\cH)
    && \text{by~\eqref{eq:reg}},
    \\ \label{eq:u-bd}
    &  \lesssim \beta(z)  \| f \|_\cH
    && \text{by~\eqref{eq:6}},
    \\     \label{eq:q-bd}
    | q |_{H^{r}(\om)}
    & =     | \grad u  |_{H^{r}(\om)}
      \lesssim \beta(z)  \|  f \|_\cH,
    && \text{by~\eqref{eq:u-bd}},
    \\ \nonumber 
    |\div\, q|_{H^r(\om)}
    & = | f - zu |_{H^r(\om)}
    \\ \nonumber 
    &\lesssim |f|_{H^r(\om)} + |z|\beta(z) \| f\|_\cH
    && \text{by~\eqref{eq:u-bd}},
    \\ \label{eq:divq-bd}
    & \lesssim \beta(z) \| f \|_\cV
    && \text{since $r \le 1$.}
  \end{align}
  Thus for all $0 \le r \le \min(p, s, 1)$, using the
  estimates~\eqref{eq:u-bd},~\eqref{eq:q-bd} and \eqref{eq:divq-bd}
  in~\eqref{eq:approx1}, we have proven~\eqref{eq:R-R_h_on_V}.

  The proof of~\eqref{eq:R-R_h_on_E} starts off as above using an
  $f \in E$. But now, due to the potentially higher regularity, we are
  able to obtain~\eqref{eq:u-bd} and~\eqref{eq:q-bd} for $r \le s_E$.
  Moreover, as in the proof of~\eqref{eq:divq-bd} above, we find that
  $|\div\, q |_{H^r(\om)} \lesssim \beta(z) \| f \|_{H^r(\om)}$. The
  argument to bound $ \| f \|_{H^r(\om)}$ by $\| f \|_\cV$ now
  requires a slight modification:  since $-\Delta f \in E$, the
  regularity estimate~\eqref{eq:reg-eig} implies
  $\| f \|_{H^{1+r}(\om)} \lesssim \| f \|_{\cH}$. Thus
  \begin{align*}
    |\div\, q|_{H^r(\om)}
    & \lesssim \beta(z) \| f \|_\cV
      && \text{ for  $r \le s_E$,}
  \end{align*}
  i.e., whenever $f \in E$, the
  estimates~\eqref{eq:u-bd},~\eqref{eq:q-bd} and \eqref{eq:divq-bd}
  hold for all $0\le r \le s_E$. Using them in~\eqref{eq:approx1}, the
  proof of~\eqref{eq:R-R_h_on_E} is complete.
\end{proof}

\begin{theorem} 
  \label{thm:total} 
  Suppose Assumption~\ref{asm:rN} (on spectral separation) and
  Assumption~\ref{asm:reg} (on elliptic regularity) hold.  Then,
  there are positive constants $C_0$ and $ h_0$ such that for all
  $h<h_0$, the FEAST iterates $\Eh \ell$ obtained using the DPG
  approximation of the resolvent converge to $E_h$ and
  \begin{align}
    \label{eq:gapEEhDPG}
  \gap_\cV (E, E_h) 
    & \le    \revv{C_0\,  h^{\min(p,s_E)},}
    \\
    \label{eq:LLhDPG}
    \hdist( \Lambda, \Lambda_h) 
    & \le    C_0\,  h^{2\min(p,s_E)}.
  \end{align}
  Here $C_0$ is independent of $h$ (but may  depend on $\beta(z_k)^{2}$, $W,$
  $C_N$, $p$, $\Lambda,$ $\Creg$, and the shape regularity of the
  mesh).
\end{theorem}
\begin{proof}
  We apply Theorem~\ref{thm:abstract}.  As we have already noted,
  Assumption~\ref{asm:Vshort} holds for the model Dirichlet problem
  with the settings in~\eqref{eq:VHA_Dirichlet}.
  Estimate~\eqref{eq:R-R_h_on_V} of Lemma~\ref{lem:rates} verifies
  Assumption~\ref{asm:Rlim}.  Thus, since
  Assumptions~\ref{asm:rN}--\ref{asm:Rlim} hold, we may now
  apply~\eqref{eq:Eh_ell_to_Eh} of Theorem~\ref{thm:abstract} to
  conclude that $\gap_\cV(\Eh\ell, E_h) \to 0$. Moreover,
  the inequality \eqref{eq:gap_E_Eh} of Theorem~\ref{thm:abstract},
  when combined with the rate estimate~\eqref{eq:R-R_h_on_E} of
  Lemma~\ref{lem:rates} at each $z_k$, proves~\eqref{eq:gapEEhDPG}.

  Finally, to prove~\eqref{eq:LLhDPG}, noting that the $\cV_h$ set
  in~\eqref{eq:Vh_DPG} satisfies Assumption~\ref{asm:Vh_in_dom(a)}, we
  appeal to \eqref{eq:dist_L_Lh} of Theorem~\ref{thm:abstract} to
  \begin{equation}
    \label{eq:DPG_L_Lh_V_H}
    \hdist( \Lambda, \Lambda_h) \lesssim  \gap_\cV(E, E_h)^2
    + 
    \gap_\cH(E, E_h)^2.
  \end{equation}
  To control the last term, first note that
  $\| e \|_\cV^2 = a(e,e) \le C_E \| e \|^2_\cH $ for all $e \in E$.
  Moreover, by Assumption~\ref{asm:Vshort},
  $ \dist{\cH} (e, E_h) \le C_\cV\dist{\cV} (e, E_h).$
  Hence
  \begin{equation}
    \label{eq:deltaHh}
    \delta^\cH_h:= \sup_{0 \ne e \in E} \frac{\dist\cH( e, E_h) }{ \| e \|_{\cH}}
  \lesssim
  {\sup_{0 \ne e \in E}}\frac{ \dist\cV( e, E_h) }{  \| e \|_{\cV}}
  \le 
  \gap_\cV(E,E_h).
  \end{equation}
  Note that 
  \[
    \gap_\cH(E, E_h) =
    \max \bigg[
    \delta_h^\cH, 
    \sup_{m \in U_{E_h}^\cH} \dist{\cH} ( m, E)
    \bigg].
  \]
  Now, by the already proved estimate of~\eqref{eq:gapEEhDPG}, we know
  that $\gap_\cV( E, E_h) \to 0$. Hence, when $h$ is sufficiently
  small, $\gap_\cV( E, E_h)<1$, so $\dim(E_h) = \dim(E) = m$. Taking
  $h$ even smaller if necessary, $\delta^\cH_h <1$
  by~\eqref{eq:deltaHh}, so by~\cite[Theorem~I.6.34]{Kato95}, there is
  a closed subspace $\tilde{E}_h \subseteq E_h$ such that
  $\gap_\cH(E, \tilde{E}_h) = \delta_h^\cH < 1.$ But this means that
  $\dim(\tilde{E}_h)= \dim(E) =  \dim(E_h) $, so $\tilde{E}_h=E_h$.
  Summarizing, for sufficiently small
  $h$, we have
  \[
    \gap_\cH(E, E_h) = \delta_h^\cH
    \lesssim \gap_\cV(E, E_h).
  \]
  Returning to~\eqref{eq:DPG_L_Lh_V_H}, we conclude that
  \[
    \hdist(\Lambda, \Lambda_h) \lesssim \gap_\cV(E, E_h)^2,
  \]
  and the proof is finished using~\eqref{eq:gapEEhDPG}. 
\end{proof}

\subsection{A generalization to additive perturbations}
\label{ssec:gener}

In this short subsection, we will 
generalize the above theory to the case of the Dirichlet operator when
perturbed
additively by a real-valued $L^\infty(\om)$ reaction term.
Let $\nu: \om \to \RRR$ be a function in $L^\infty(\om)$ and let
\begin{equation}
  \label{eq:a-perturbed}
  a(u, v) = \int_\om \big[
  \grad u \cdot \grad \bar v - \nu u \bar v \big]
  \; dx  
\end{equation}
for any $ u, v \in \dom(a) = \cV = H_0^1(\om)$. The operator under
consideration in this subsection is the unbounded selfadjoint operator
$A$ on $\cH = L^2(\om)$ generated by the form $a$, per a standard
representation theorem~\cite[Theorem~10.7]{Schmu12}.

The starting point for our theory in the previous subsections was an
inf-sup condition (see Lemma~\ref{lem:dirichlet}) for the resolvent
form $ b(u, v) = z(u, v)_\cH - a(u,v).  $ We claim that
Lemma~\ref{lem:dirichlet} can be extended to the new $a(\cdot,\cdot)$.
To prove the claim, given any $v \in H_0^1(\om),$ we construct a
$w \in H_0^1(\om)$ slightly differently from the proof of
Lemma~\ref{lem:dirichlet}, namely
\[
  \revv{w = R(\bar z) \, ( \bar z v + \nu v),}
\]
\rev{which solves $b(s,w) = z(s,v)_{\mathcal{H}} + (\nu s, v)_{\mathcal{H}}$ for all $s \in H_0^1(\Omega)$.} Then we continue to obtain
the identity
\begin{equation}
  \label{eq:8}
b(v, v-w) =
-|v|_{H^1(\om)}^2.  
\end{equation}
Next, for any $\mu > \| \nu\|_{L^\infty(\om)}$, the form domain
$\dom(a) = H_0^1(\om)$ equals 
$\dom(A+ \mu)^{1/2}$ by~\cite[Proposition~10.5]{Schmu12}. The same
result also gives
\[
  a(u,v) = ( (A+\mu)^{1/2} u, (A+\mu)^{1/2} v)_\cH - \mu (u,v)_\cH
  \qquad \revv{\text{for all}} \: u, v \in H_0^1(\om).
\]
Hence
\begin{equation}
  \label{eq:9}
  |w|_{H^1(\om)}^2 = a(w,w) + (\nu w, w)_\cH \le
  a(w,w) + \mu \| w \|_\cH^2
  = \|(A + \mu)^{1/2}w\|_\cH^2.
\end{equation}

To proceed, recall that for any $z$ in the resolvent set,  functional calculus
\cite[Theorem~6.4.1]{BuhleSalam18} shows that the spectrum of the normal
operator $(A+\mu)^{1/2} R(z)$, consists of
$\{(\lambda+\mu)^{1/2} / (z - \lambda): \lambda \in \Sigma(A)\}$.
Thus $(A+\mu)^{1/2} R(z)$ is a bounded operator  of norm
$ c_z = \sup\{ |\lambda+\mu|^{1/2} / |z - \lambda|: \;\lambda \in
\Sigma(A)\} < \infty.$
Hence~\eqref{eq:9} implies  
$
|w|_{H^1(\om)}
  \le  \| (A+\mu)^{1/2} R(\bar z) \, (  \bar z v + \nu v)\|_\cH
  \le  c_z \|  \bar z v + \nu v\|_{\cH}.
  $
  Using the \Poincare\ inequality $c_P\| v
  \|_\cH \le  |v|_{H^1(\om)}$,  this implies
  $|w|_{H^1(\om)} \le (|z| + \mu) (c_z/c_P) |v|_{H^1(\om)}, $
so
\begin{equation}
  \label{eq:7}
  \revv{
  |v - w|_{H^1(\om)}   \le d(z) |v|_{H^1(\om)},
  }
\end{equation}
where
$
d(z) =  1 + (|z|+ \mu) c_z/c_P.  
$
Combining~\eqref{eq:8} and \eqref{eq:7}, we have 
\begin{align*}
\sup_{y \in H_0^1(\om)} 
  \frac{| b(v,y) |}{ \quad| y |_{H^1(\om)}}\geq
  \frac{| b(v, v-w) |}{ \quad| v-w  |_{H^1(\om)}}
  \ge
  \frac{ |v|_{H^1(\om)}^2}{ d(z)| v|_{H^1(\om)}},
\end{align*}
so the inf-sup condition follows, extending
Lemma~\ref{lem:dirichlet} as claimed.

\begin{lemma}[Generalization of Lemma~\ref{lem:dirichlet}]
  Suppose $a$ as in~\eqref{eq:a-perturbed},
  $ b(u, v) = z(u, v)_\cH - a(u,v)$, $z$ is in the resolvent set
  of $A$, and  $d(z)$ is as defined above. Then for all
  $v \in H_0^1(\om)$,
  \begin{align*}
    \sup_{y \in H_0^1(\om)} 
    \frac{| b(v,y) |}{ \quad| y |_{H^1(\om)}}\geq
    d(z)^{-1} \,| v |_{H^1(\om)}.
  \end{align*}
\end{lemma}

Using this lemma in place of Lemma~\ref{lem:dirichlet}, the remainder
of the analysis proceeds with minimal changes, provided we also assume
that $\nu$ is piecewise constant. More precisely, assume that $\nu$ is
constant on each element of the mesh $\oh$. Then the same Fortin
operator used in the proof of Lemma~\ref{lem:dpg-resolvent}
applies. Hence the final result of Theorem~\ref{thm:total} holds with
a possibly different constant $C_0$ (still independent of $h$)
whenever Assumption~\ref{asm:reg} holds.

\section{Numerical convergence studies}\label{NS}

In this section, we report on our numerical convergence studies using
the FEAST algorithm with the DPG discretization for the model
Dirichlet eigenproblem.  This spectral approximation technique is
exactly the one described in Section~\ref{ssec:dpg}.  An
implementation of this technique was built using~\cite{Gopal17}, which
contains a hierarchy of Python classes representing approximations of
spectral projectors. The DPG discretization is implemented using a
python interface into an existing well-known C++ finite element
library called NGSolve~\cite{Schob17}.  We omit the implementation
details of the FEAST algorithm as they can be found either in our
public code~\cite{Gopal17} or previous works like
\cite[Algorithm~1.1]{Saad16} and \cite{GuttePolizTang15}. We note that
our implementation performs an implicit orthogonalization through a
small Rayleigh-Ritz eigenproblem at each iteration.  For all
experiments reported below, we set $r_N$ to  the rational function
corresponding to the Butterworth filter obtained by setting
$w_N=0$ and
\begin{align}\label{CircleQuad}
z_k=\gamma e^{\ii
  (\theta_k+\phi)}+y,\quad\quad w_k=\gamma
  e^{\ii (\theta_k+\phi)}/N, \qquad
  k=0, \ldots, N-1,
\end{align}
where $\theta_k=2\pi k/N$ and $ \phi=\pm\pi/N.$ This corresponds to an
approximation of the contour integral in~\eqref{eq:DunfodTaylor}, with
a circular contour $\Gamma$ of radius $\gamma$ centered at $y$, using
the trapezoidal rule with $N$ equally spaced quadrature points. In all
experiments reported below, we set $N=8$.

\subsection{Discretization \revv{errors on the unit square}}


\begin{figure}[b]
  \begin{subfigure}[t]{0.49\textwidth} 
    \begin{center}  
      \begin{tikzpicture}
        \begin{loglogaxis}[
          footnotesize,
          width=0.95\textwidth,
          height=\textwidth,
          xlabel=$h$,
          ylabel={$d_h$},
          legend pos = south east,
          max space between ticks=30pt,
          ]


          \addplot coordinates {
            (0.2500000000, 6.747273e+00)
            (0.1250000000, 1.975461e+00)
            (0.0625000000, 7.534260e-01)
            (0.0312500000, 3.509711e-01)
            (0.0156250000, 1.723048e-01)
            (0.0078125000, 8.576663e-02)
          };

          \addplot coordinates {
            (0.2500000000, 9.049084e-01)
            (0.1250000000, 3.790549e-01)
            (0.0625000000, 8.318753e-02)
            (0.0312500000, 1.879642e-02)
            (0.0156250000, 4.437640e-03)
            (0.0078125000, 1.077704e-03)
          };

          \addplot coordinates {
            (0.2500000000, 3.705158e-01)
            (0.1250000000, 7.932601e-02)
            (0.0625000000, 9.720224e-03)
            (0.0312500000, 1.209351e-03)
            (0.0156250000, 1.508522e-04)
            (0.0078125000, 1.886285e-05)
          };
          
          \logLogSlopeTriangle{0.59}{0.15}{0.33}{3}{black};
          \logLogSlopeTriangle{0.59}{0.15}{0.515}{2}{black};
          \logLogSlopeTriangle{0.59}{0.15}{0.705}{1}{black};

          \legend{$p=1$, $p=2$, $p=3$}
        \end{loglogaxis}
      \end{tikzpicture}
    \caption{Convergence rates for eigenfunctions}
    \label{fig:efsqr}
    \end{center}  
  \end{subfigure}
  \begin{subfigure}[t]{0.49\textwidth} 
    \begin{center}  
      \begin{tikzpicture}
        \begin{loglogaxis}[
          footnotesize,
          width=0.95\textwidth,
          height=\textwidth,
          xlabel=$h$,
          ylabel={$\hdist( {\Lambda}, {\Lambda_h})$},
          legend pos = south east,
          max space between ticks=30pt,
          ]


          \addplot coordinates {
            (0.2500000000, 1.455193e+01)
            (0.1250000000, 4.124450e+00)
            (0.0625000000, 9.859321e-01)
            (0.0312500000, 2.436991e-01)
            (0.0156250000, 6.066035e-02)
            (0.0078125000, 1.513589e-02)
          };

          \addplot coordinates {
            (0.2500000000, 5.419321e-01)
            (0.1250000000, 5.954395e-02)
            (0.0625000000, 4.126409e-03)
            (0.0312500000, 2.647773e-04)
            (0.0156250000, 1.668255e-05)
            (0.0078125000, 1.045518e-06)
          };

          \addplot coordinates {
            (0.2500000000, 1.472728e-02)
            (0.1250000000, 5.240445e-04)
            (0.0625000000, 7.863915e-06)
            (0.0312500000, 1.218536e-07)
            (0.0156250000, 1.896943e-09)
            (0.0078125000, 3.102940e-11)
          };

          \logLogSlopeTriangle{0.59}{0.15}{0.31}{6}{black};
          \logLogSlopeTriangle{0.59}{0.15}{0.55}{4}{black};
          \logLogSlopeTriangle{0.59}{0.15}{0.76}{2}{black};

          \legend{$p=1$, $p=2$, $p=3$}
        \end{loglogaxis}
      \end{tikzpicture}
      \caption{Convergence rates for eigenvalues}
      \label{fig:ewsqr}
    \end{center}  
    \end{subfigure}
    \caption{Results for the unit square}
    \label{fig:ewf}
\end{figure}
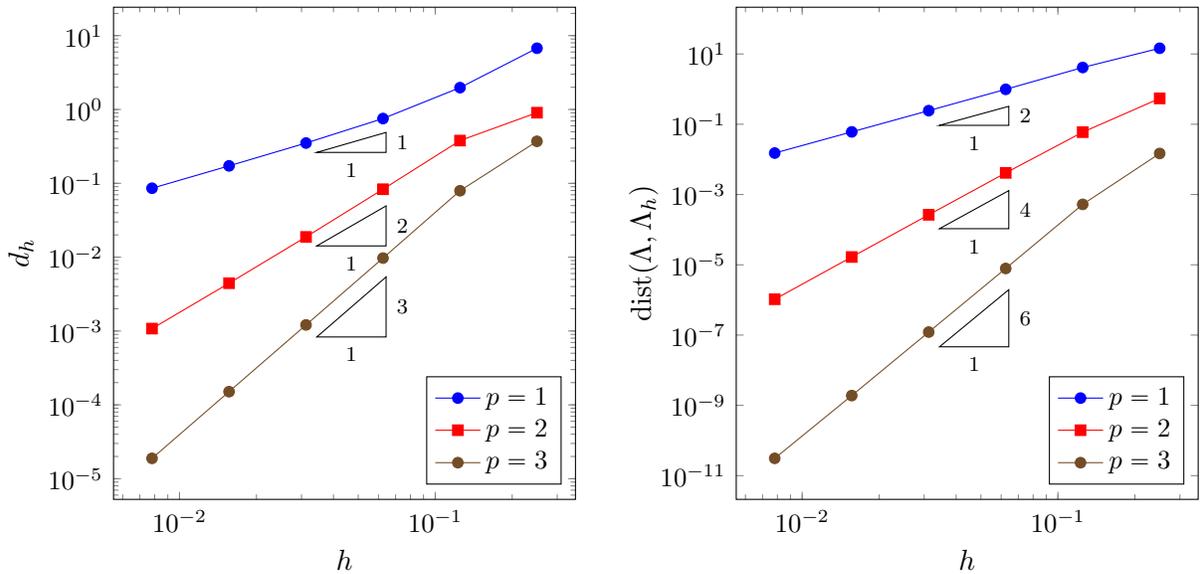

Let $\Omega = (0,1) \times (0,1)$ and consider the Dirichlet
eigenvalues enclosed within the circular contour $\Gamma$ of radius
$\gamma = 45$ and center $y = 20$.  The exact set of eigenvalues for
this example is known to be $\Lambda = \{ 2\pi^2, 5 \pi^2\}$. The
first eigenvalue $ 2\pi^2= \lambda_1$ is of multiplicity 1, while the
second $5\pi^2 = \lambda_2 = \lambda_3$ is of multiplicity~2.
The corresponding eigenfunctions are well-known analytic functions.

To perform the numerical studies, we begin by solving our problem on a
coarse mesh of mesh size $h = 2^{-2}$ and refine until we reach a mesh
size of $h = 2^{-7}$. Each mesh refinement halves the mesh size by
either bisecting or quadrisecting the triangular elements of a mesh.
For each mesh size value of $h$, we perform this experiment for
polynomial degrees $p = 1, 2, $ and $3$.  After each experiment we
collect the approximate eigenvalues ordered so that 
$\lambda_{1, h} \le \lambda_{2, h} \le \lambda_{3, h}$
and their corresponding eigenfunctions $e_{i, h}$.

One way to measure the convergence of eigenfunctions is through
\begin{align*}
  \delta_i^{(1)} =& \min_{0 \ne e \in E} |e_{i,h} - e|_{H^1(\Omega)} =
              \dist{H_0^1(\om)}( e_{i,h}, E),
  \\
  \delta_i^{(2)} =& \min_{0 \ne e_h \in E_h} |e_i - e_h|_{H^1(\Omega)} =
              \dist{H_0^1(\om)}( e_i, E_h).
\end{align*}
Note that both $\delta_i^{(1)}$ and $\delta_i^{(2)}$ are bounded by
$\gap_{H_0^1(\om)} (E_h, E)$.  Since computing $\delta_i^{(1)}$ and
$\delta_i^{(2)}$ require exact integration of quantities involving the
exact eigenspace, we instead compute
\[
  \delta_{i, h}^{(1)} = \dist{H_0^1(\om)} (e_{i, h}, I_h E)
  \quad
  \revv{\text{and}}
  \quad
  \delta_{i, h}^{(2)} = \dist{H_0^1(\om)} (I_h e_i, E_h),
\]
where $I_h$ is a standard interpolant into the finite element space
$\cV_h$.  For brevity, instead of plotting the behavior of each
$\delta_{i, h}^{(j)}$ for all $i, j$, we plot the behavior of their
sum
\[
  d_h = \sum_{i=1}^3 \sum_{j=1}^2 \delta_{i, h}^{(j)}
\]
for decreasing mesh sizes $h$ and increasing polynomial degrees $p$ in
Figure~\ref{fig:ewf}. 
In the same figure panel, we also display the
observed errors in the computed eigenvalues in $\Lambda_h$ by plotting
the Hausdorff distance $\hdist(\Lambda, \Lambda_h)$ for various values
of $h$ and~$p$.

Since $\delta_i^{(j)}$ should go to zero at the same rate as
$\gap_{H_0^1(\om)} (E_h, E)$ and since the interpolation errors are of
the same order as the gap, we expect $d_h$ to go to zero as $h\to 0$
at the same rate as $\gap_{H_0^1(\om)} (E_h, E).$ From
Figure~\ref{fig:efsqr}, we observe that $d_h$ appears to converge to 0
at the rate $O(h^{p})$ for $p=1, 2,$ and 3.  Since the eigenfunctions
on the unit square are analytic, Assumption~\ref{asm:reg} holds for
this example with {\em any} $s_E > 0$. Therefore, our observation on
the rate of convergence of $d_h$ is in agreement with the gap
estimate~\eqref{eq:gapEEhDPG} of Theorem~\ref{thm:total}.
Figure~\ref{fig:ewsqr} shows that as $h$ decreases,
$\hdist(\Lambda, \Lambda_h)$ decreases to 0 at the rate $O(h^{2p})$
for $p=1, 2,$ and 3.  This is also in good agreement with the
eigenvalue error estimate~\eqref{eq:LLhDPG} of
Theorem~\ref{thm:total}.

The results presented above using the DPG discretization are
comparable to those found in~\cite{GopalGrubiOvall18} using the FEAST
algorithm with the standard finite element discretization of
comparable orders.

\begin{remark} \label{rem:other_experim}
In other unreported experiments, we found that setting
$Y_h$ to
\[
  \tilde{Y}_h=\{ y \in H^1(\oh): y|_K \in P_{p+1}(K)\}
\]
also gave the same convergence rates. This indicates that the space
dictated by the theory, namely
$Y_h = \{ y \in H^1(\oh): y|_K \in P_{p+3}(K)\}$, might be overly
conservative. We already noted one approach to improve the estimates
in Remark~\ref{rem:Pih}.  Another approach might be through a
perturbation argument, as the theory in~\cite{DemkoGopal13a} proves
the error estimate of Lemma~\ref{lem:dpg-resolvent} at $z=0$ even when
$Y_h$ is replaced by $\tilde{Y}_h$.
\end{remark}

\subsection{Convergence rates on an L-shaped domain}

\begin{table}
	\centering
	\begin{tabular}{|c|cc|cc|cc|}
		\hline
		         & $\lambda_1$ &      & $\lambda_2$ &      & $\lambda_3$ & 			\\	
		     $h$ &         ERR &  NOC &         ERR &  NOC &         ERR & NOC		\\	\hline
		$2^{-2}$ &    6.29e-02 &  --- &    3.29e-02 &  --- &    5.95e-02 &  ---	\\
		$2^{-3}$ &    2.41e-02 & 1.39 &    2.65e-03 & 3.63 &    4.05e-03 & 3.88	\\
		$2^{-4}$ &    9.48e-03 & 1.34 &    2.55e-04 & 3.38 &    2.59e-04 & 3.97	\\
		$2^{-5}$ &    3.75e-03 & 1.34 &    2.99e-05 & 3.09 &    1.63e-05 & 3.99	\\
		$2^{-6}$ &    1.49e-03 & 1.34 &    4.03e-06 & 2.89 &    1.02e-06 & 4.00	\\	\hline
	\end{tabular}
	\caption{Eigenvalue errors (ERR) and numerical order of convergence (NOC) for the
          smallest three eigenvalues on the L-shaped domain.}
        \label{tab:L}
\end{table}

%

In this example, we consider the Dirichlet eigenvalues of the L-shaped
domain $\Omega = (0,2) \times (0,2) \setminus [1,2] \times [1,2]$
enclosed within a circular contour of radius $\gamma = 8$ centered at
$y = 15$.  The first three Dirichlet eigenvalues are enclosed in this
contour and we are interested in determining the eigenvalue error and
numerical order of convergence for these. We use the results reported
in \cite{TrefeBetcke06} as our reference eigenvalues, namely
$\lambda_1 \approx 9.6397238$, \revv{$\lambda_2 \approx 15.197252$}, and
$\lambda_3 = 2\pi^2$.

With the above values of $\lambda_i$ (displayed up to the digits the
authors of \cite{TrefeBetcke06} claimed confidence in), we define
$\text{ERR}(h) = |\lambda_{i,h} - \lambda_i|$, where
$\lambda_{1,h}\le \lambda_{2,h}\le \lambda_{3,h}$ are the approximate
eigenvalues obtained by FEAST. Then we define the numerical order of
convergence (NOC) as
$\text{NOC}(h) = \log(\text{ERR}(2h)/\text{ERR}(h)) / \log(2)$.

We perform our convergence study, as in the unit square case, using a
sequence of uniformly refined meshes, starting from a mesh size of
$h = 2^{-2}$ and ending with a mesh size of $h = 2^{-6}$. In this
example we confine the scope of our convergence study to polynomial
degree $p = 2$.  Further mesh refinements or higher degrees are not
studied because the exact eigenvalues are only available to limited
precision and errors below this precision cannot be used to surmise
convergence rates accurately. The observations are compiled in
Table~\ref{tab:L}.

From the first column of Table~\ref{tab:L}, we find that the first
eigenvalue is observed to converge at a rate of approximately $4/3$.
For polygonal domains, its well known that Assumption~\ref{asm:reg}
holds with any positive $s$ less than the $\pi/\alpha$ where $\alpha$
is the largest of the interior angles at the vertices of the
polygon. Clearly $\alpha = 3 \pi/2$ for our L-shaped $\om$.  The
eigenfunction corresponding to the first eigenvalue is known to be
limited by this regularity, so $s_E$ may be chosen to be any positive
number less than $2/3$.  Therefore, the observed convergence rate of
$4/3$ for the first eigenvalue is in agreement with the rate of
$2\min(p,s_E)$ established in Theorem~\ref{thm:total}.  Although
Theorem~\ref{thm:total} does not yield improved convergence rates for
the other eigenvalues, which have eigenfunctions of higher regularity,
we observe from the remaining columns of Table~\ref{tab:L} that in
practice we do observe higher order convergence rates. E.g., the
eigenfunction corresponding to $\lambda_3 = 2\pi^2$ is analytic and we
observed that the corresponding eigenvalue converges at a rate
$O(h^{2p})$ that is not limited by $s_E$.

\section{Application to optical fibers}\label{fib}

Double-clad step-index optical fibers have resulted in numerous
technological innovations. Although originally intended to carry
energy in a single mode, for increased power operation large mode area
(LMA) fibers are now being sold extensively. LMA fibers usually have
multiple guided modes. In this section, we show how to use the method
we developed in the previous sections to compute such modes. We begin
by showing that the problem of computing the fiber modes can be viewed as
a problem of computing an eigenvalue cluster of an operator of the
form discussed in Subsection~\ref{ssec:gener}.


These optical fibers have a cylindrical core of radius $\rcore$ and a
cylindrical cladding region enveloping the core, extending to radius
$\rclad$. We set up our axes so that the longitudinal direction of the
fiber is the $z$-axis. The transverse coordinates will be denoted
$x, y$ while using Cartesian coordinates and the eigenvalue problem
will be posed in these coordinates. Thus the space dimension
(previously denoted by $n$) will be fixed to $2$ in this section, so
denoting the refractive index of the fiber by $n$ in this section
causes no confusion.  We have in mind fibers whose refractive index
$n(x,y)$ is a piecewise constant function, equalling $\ncore$ in the
core, and $\nclad$ in the cladding region $(\nclad < \ncore)$. The
guided modes, also called the transverse core modes, decay
exponentially in the cladding region.

These {\em modes} of the fiber, which we denote by $\varphi_l(x,y),$ are
non-trivial functions that, together with their accompanying
(positive) {\em propagation constants} $\beta_l,$ solve
\begin{subequations}
  \label{eq:optic-ewp}
\begin{equation}
  \label{eq:optic-ewp-pde}
  (\Delta + k^2 n^2) \varphi_l = \beta_l^2 \varphi_l,
  \qquad r < \rcore,
\end{equation}
where $k$ is a given wave number of the signal light,
$\Delta = \partial_{xx} + \partial_{yy}$ denotes the Laplacian in the
transverse coordinates $x, y$. Since the guided modes decay
exponentially in the cladding, and since the cladding radius is
typically many times larger than the core, we
supplement~\eqref{eq:optic-ewp-pde} with zero Dirichlet boundary
conditions at the end of the cladding:
\begin{equation}
  \label{eq:optic-ewp-bc}
  \varphi_l = 0, \qquad r = \rcore.
\end{equation}
\end{subequations}
Since the spectrum of the Dirichlet operator $\Delta$ lies in the
negative real axis and has an accumulation point at $-\infty$, we
expect to find only finitely many $ \lambda_l \equiv \beta_l^2>0$
satisfying~\eqref{eq:optic-ewp}. This finite collection of eigenvalues
$\lambda_l$ form our eigenvalue cluster $\Lambda$ in this application,
and the corresponding eigenspace $E$ is the span of the modes
$\varphi_l$.

From the standard theory of step-index fibers \cite{Reide16}, it
follows that the propagation constants $\beta_l$ of guided modes
satisfy
\[
  \rev{\nclad^2} k^2 < \beta_l^2  < \rev{\ncore^2} k^2.
\]
Thus, having a pre-defined search interval, the computation of the
eigenpairs $(\lambda_l, \varphi_l)$ offers an example very well-suited
for applying the FEAST algorithm.  Moreover, since separation of
variables can be employed to calculate the exact solution in terms of
Bessel functions, we are able to perform convergence studies as well.
Below, we apply the algorithm to a realistic fiber using the
previously described DPG discretization of the resolvent of the
Helmholtz operator $\Delta + k^2 n^2$ with Dirichlet boundary
conditions to a realistic fiber.

The fiber we consider is the commercially available ytterbium-doped
Nufern\texttrademark \;(nufern.com) fiber, whose typical  parameters are 
\begin{equation}
  \label{eq:Nufern}
  \ncore=1.45097, \;\;
  \nclad = 1.44973,\;\;
  \rcore = 0.0125~\mathrm{m},\;\;
  \rclad=16\rcore.
\end{equation}
The typical operating wavelength for signals input to this fiber is
$1064$ nanometers, so we set the wavenumber to
$k = (2 \pi/1.064)\times 10^{6}$. Due to the small fiber radius, we
compute after scaling the eigenproblem~\eqref{eq:optic-ewp} to the
unit disc $\hat \om = \{ r<1\}$, i.e., we compute modes
$\hat\varphi_l: \hat \om \to \CCC$ satisfying
$ (\Delta + k^2 n^2 \rclad^2) \hat\varphi_l = \rclad^2\beta_l^2
\hat\varphi_l$ in $\hat\om$ and $\hat\varphi_l = 0$ on
$\partial\hat\om$.  As in the previous section, all results here are
generated using our code~\cite{Gopal17} built atop
NGSolve~\cite{Schob17}. Note that all experiments in this section are
performed using the reduced $\tilde{Y}_h$ mentioned
in Remark~\ref{rem:other_experim}.

Results from the computation are given in Figures~\ref{fig:ybmesh}
and~\ref{fig:ybeigenfuncs}.  Note that the elements whose boundary
intersects the core or cladding boundary are isoparametrically curved
to minimize boundary representation errors -- see
Figures~\ref{fig:ybmesh-far} and~\ref{fig:ybmesh-near}. The modes are
localized near the core region, so the mesh is designed to be finer
there. A six dimensional eigenspace was found. The computed basis for
the 6-dimensional space of modes, obtained using polynomial degree
$p=6$, are shown (zoomed in near the core region) in the plots of
Figure~\ref{fig:ybeigenfuncs}. The mode $e_6$  shown in
Figure~\ref{fig:mode6} is considered the ``fundamental mode'' for this
fiber, also called the LP01 mode in the optics literature~\cite{Reide16}.

\begin{figure}
	\centering
	\begin{subfigure}[t]{0.45\textwidth}
		\centering
		\includegraphics[width=0.9\textwidth]{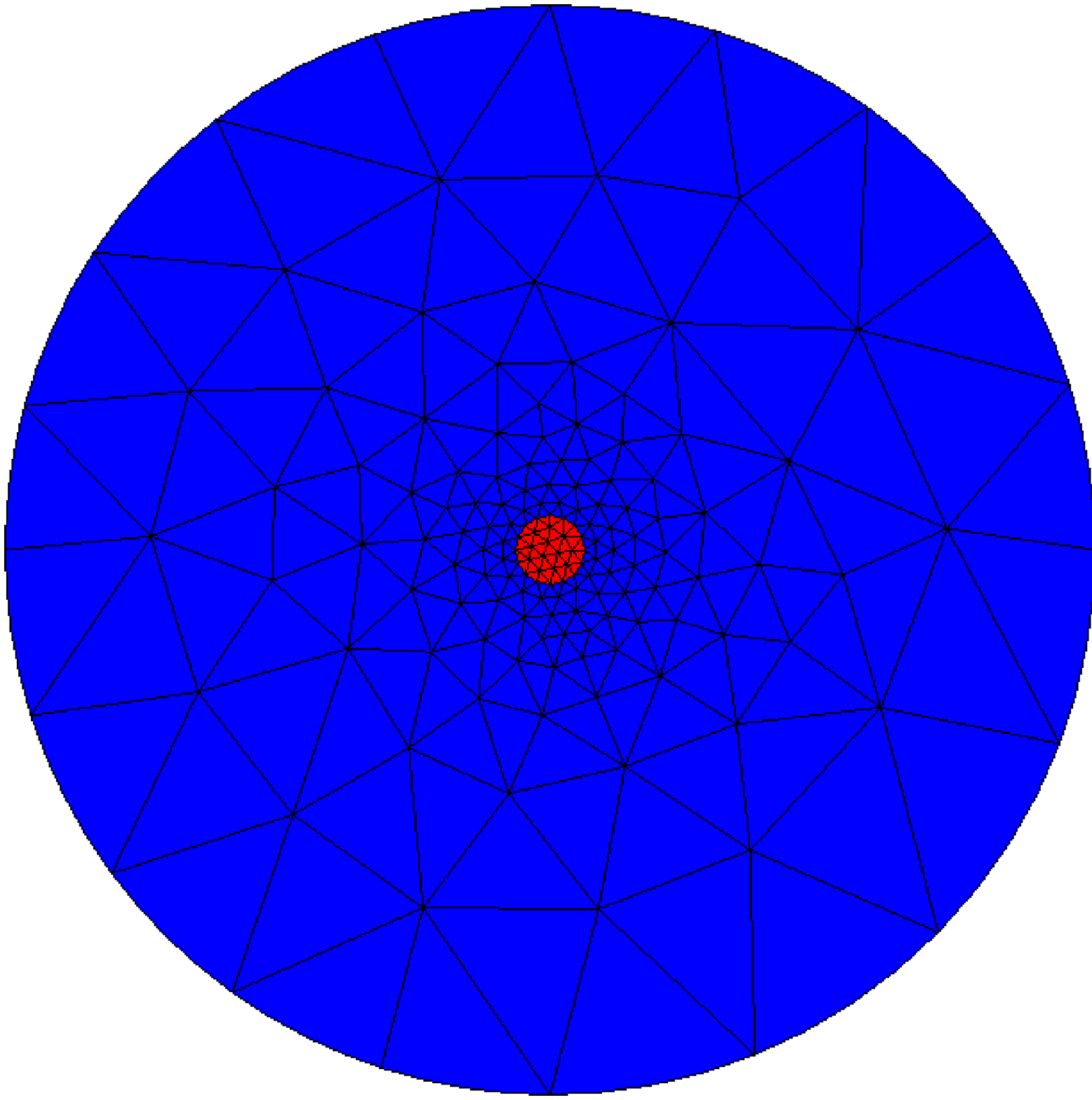}
		\subcaption{The mesh with curved
                  elements adjacent to the core and cladding boundaries.}
                \label{fig:ybmesh-far}
	\end{subfigure}%
	\quad 
	\begin{subfigure}[t]{0.45\textwidth}
		\centering
		\includegraphics[width=0.9\textwidth]{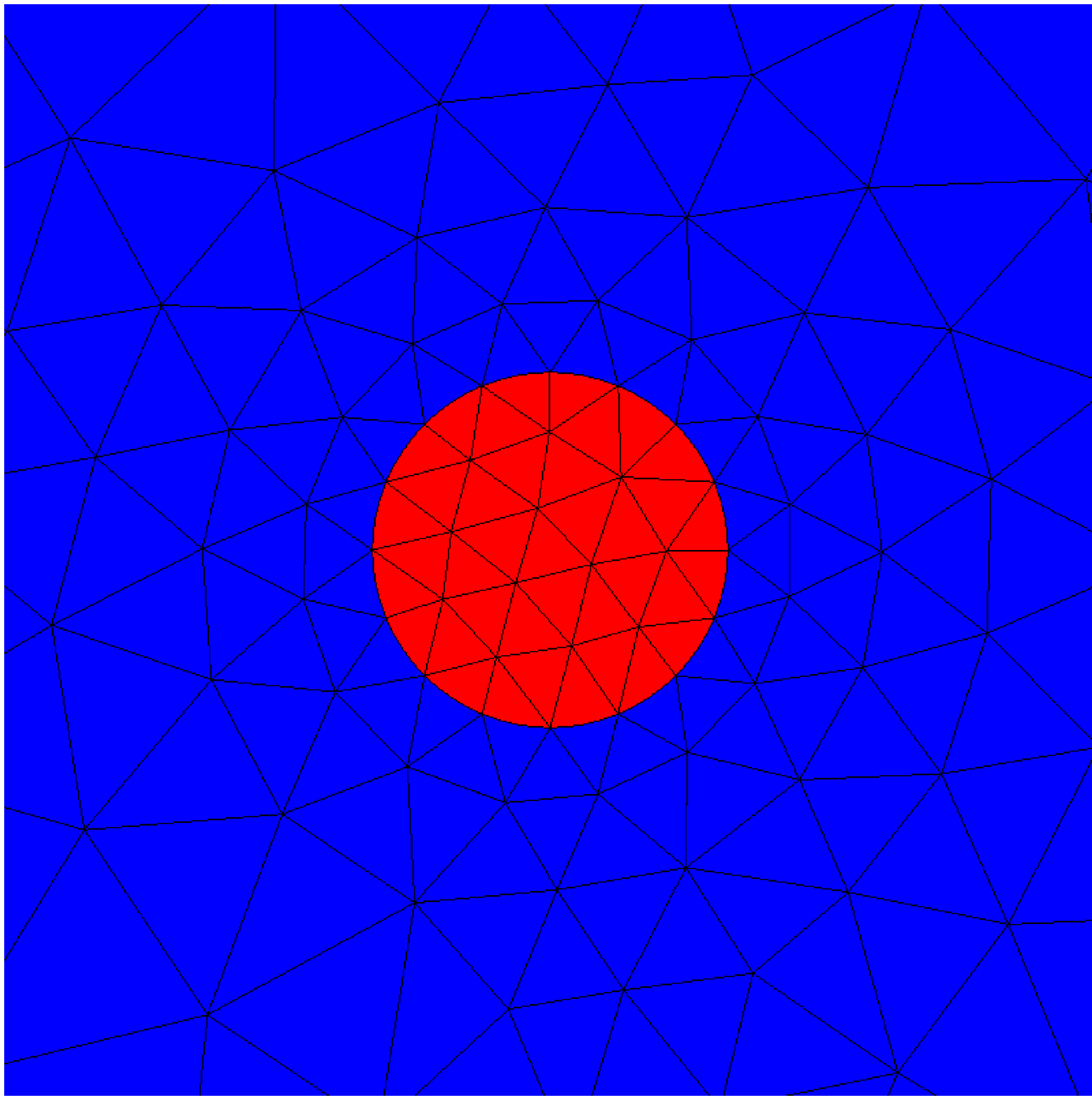}
		\subcaption{Zoomed-in view of the mesh in
                  Figure~\ref{fig:ybmesh-far} near
                  the core.}
                \label{fig:ybmesh-near}
	\end{subfigure}%
	\caption{The mesh used for computing modes of the ytterbium-doped fiber.}
	\label{fig:ybmesh}
\end{figure}

\begin{figure}
	\centering
	\begin{subfigure}[t]{0.3\textwidth}
		\centering
		\begin{tikzpicture}
		    \node[anchor=south west,inner sep=0] (image) at (0,0) {\includegraphics[width=0.875\textwidth]{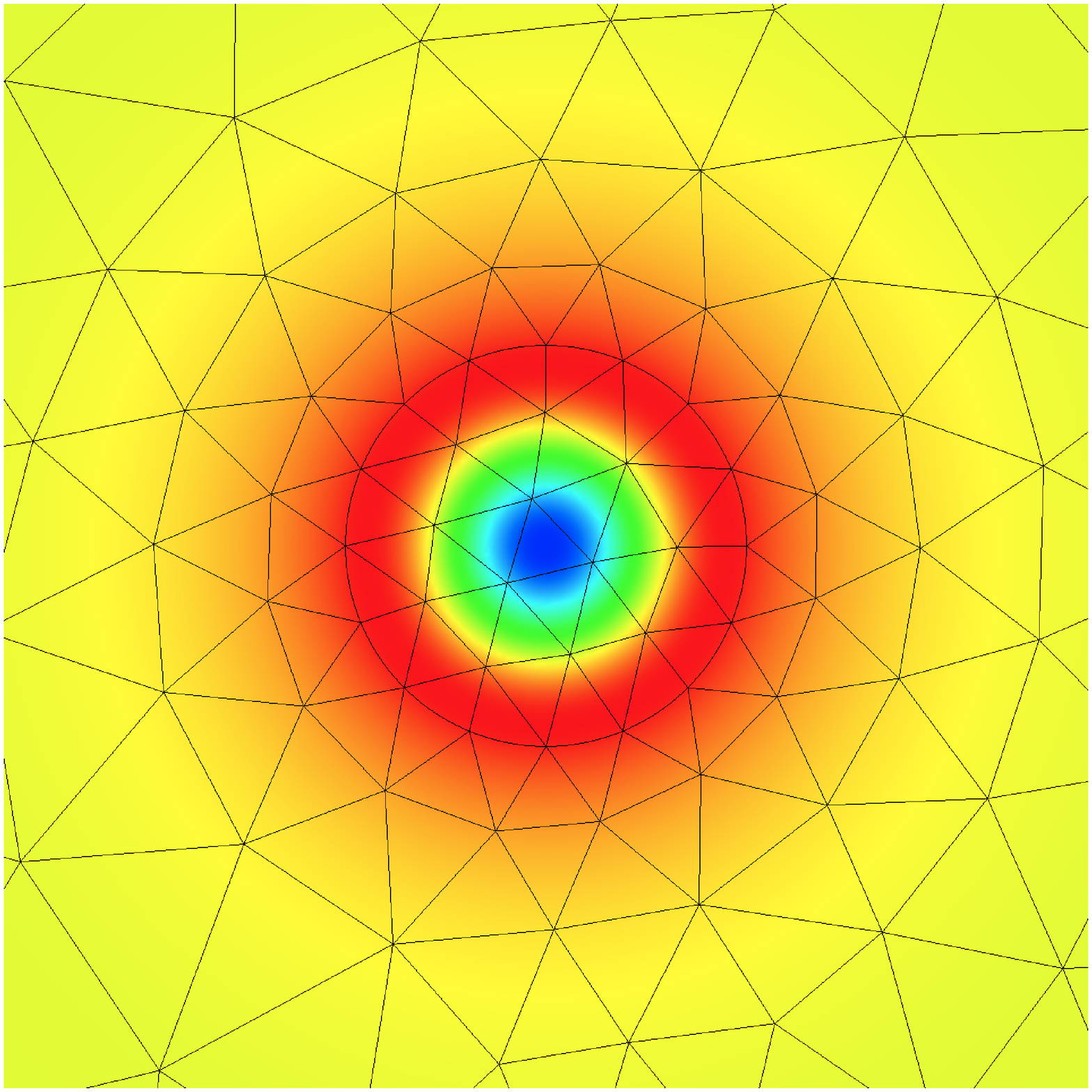}};
		    \begin{scope}[x={(image.south east)},y={(image.north west)}]
		        \draw[black, opacity=.5, thick, dashed] (0.5,0.5) circle (0.19);
		    \end{scope}
		\end{tikzpicture}
		\subcaption{$ \varphi^h_1$}
	\end{subfigure}%
	~ 
	\begin{subfigure}[t]{0.3\textwidth}
		\centering
		\begin{tikzpicture}
		    \node[anchor=south west,inner sep=0] (image) at (0,0) {\includegraphics[width=0.875\textwidth]{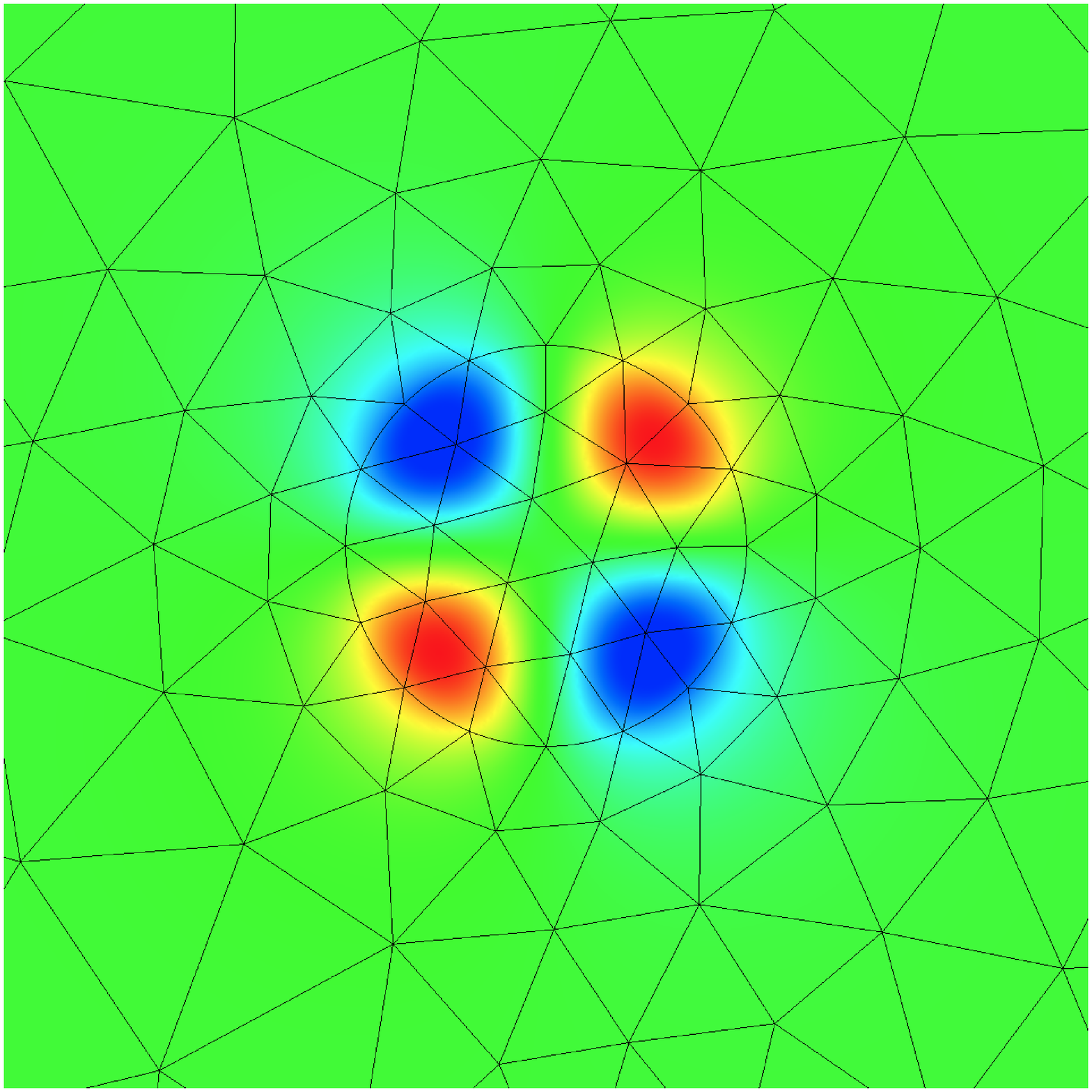}};
		    \begin{scope}[x={(image.south east)},y={(image.north west)}]
		        \draw[black, opacity=.5, thick, dashed] (0.5,0.5) circle (0.19);
		    \end{scope}
		\end{tikzpicture}
		\subcaption{$ \varphi^h_2$}
	\end{subfigure}%
	~ 
	\begin{subfigure}[t]{0.3\textwidth}
		\centering
		\begin{tikzpicture}
		    \node[anchor=south west,inner sep=0] (image) at (0,0) {\includegraphics[width=0.875\textwidth]{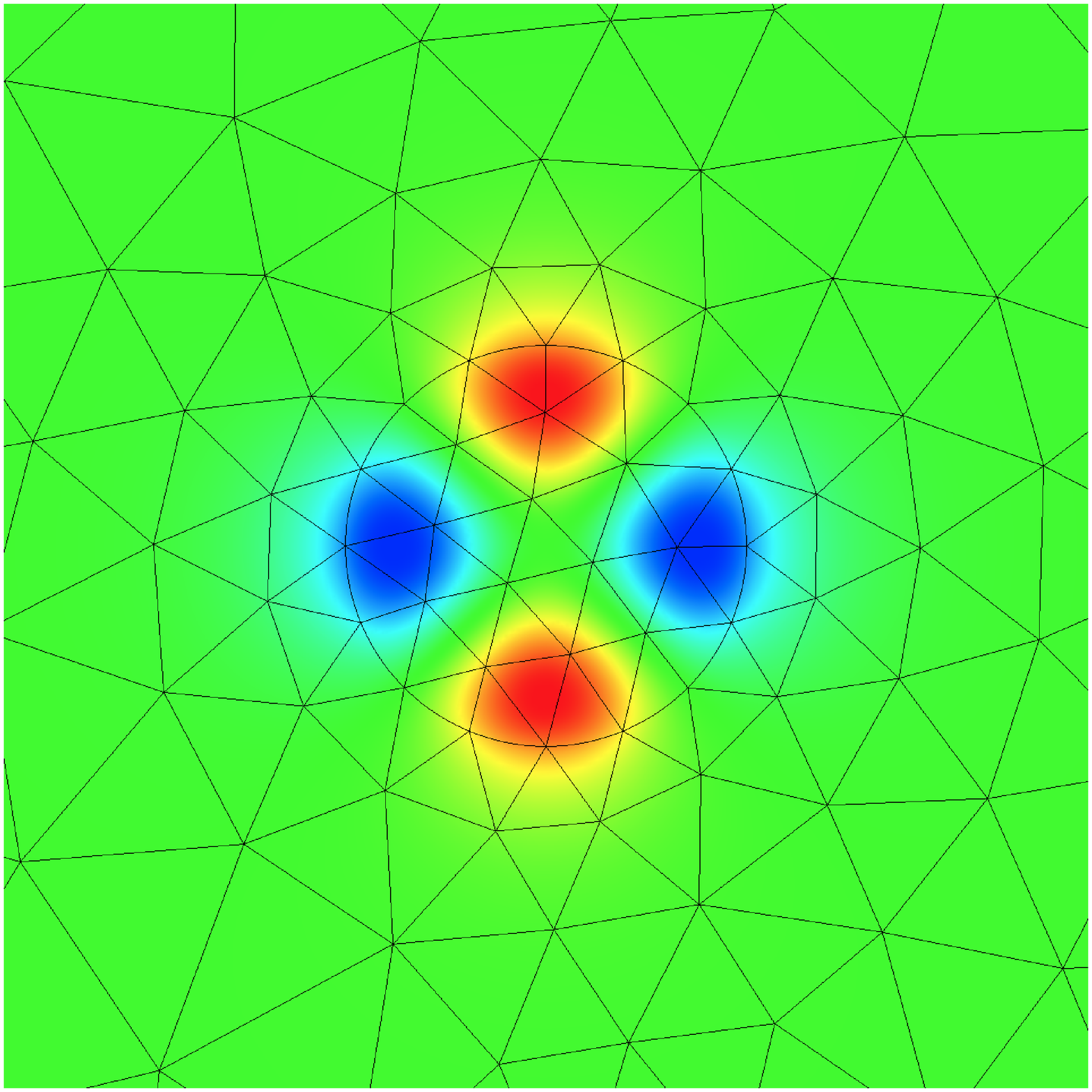}};
		    \begin{scope}[x={(image.south east)},y={(image.north west)}]
		        \draw[black, opacity=.5, thick, dashed] (0.5,0.5) circle (0.19);
		    \end{scope}
		\end{tikzpicture}
		\subcaption{$\varphi^h_3$}
	\end{subfigure}
	\newline
	\newline
	\noindent
	\begin{subfigure}[t]{0.3\textwidth}
		\centering
		\begin{tikzpicture}
		    \node[anchor=south west,inner sep=0] (image) at (0,0) {\includegraphics[width=0.875\textwidth]{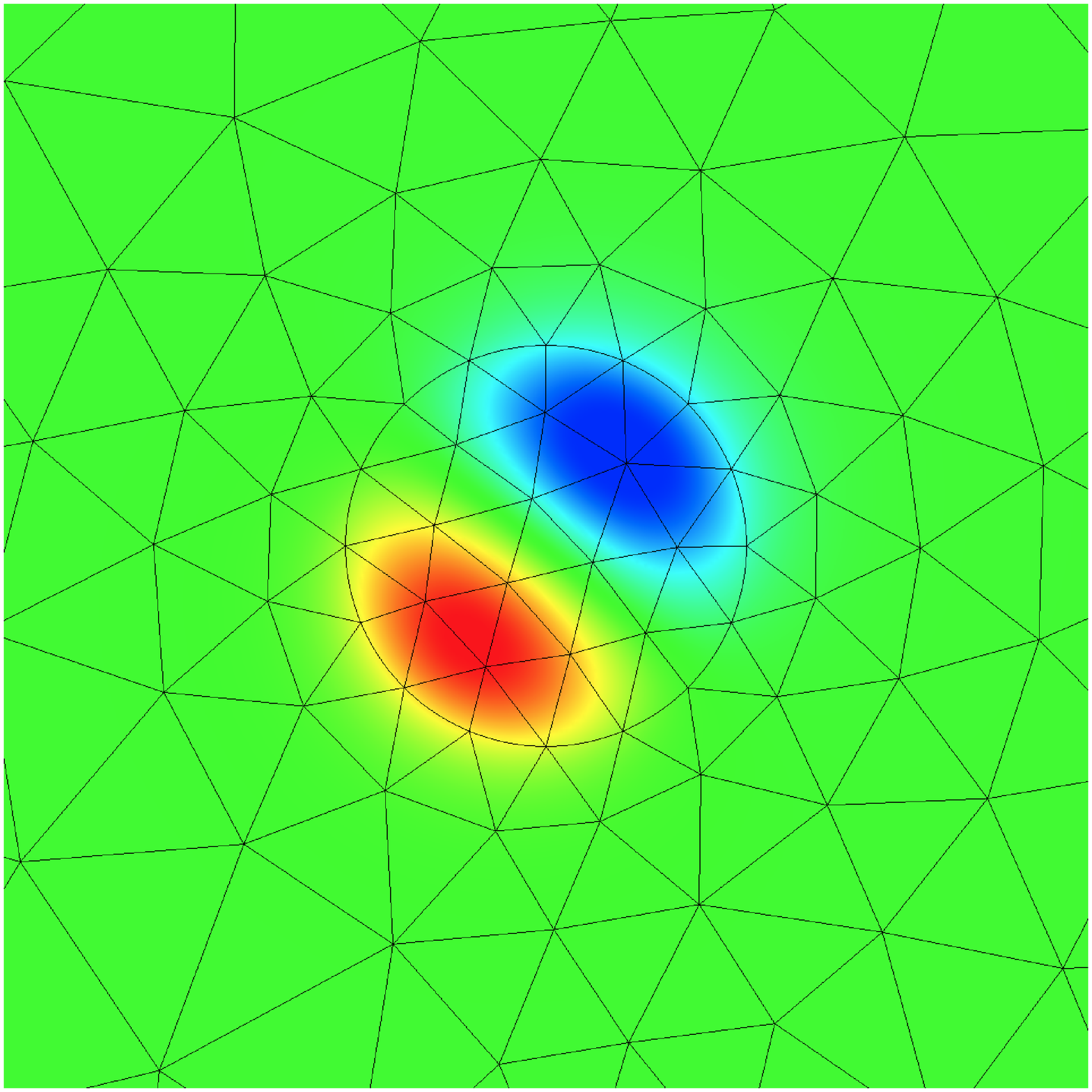}};
		    \begin{scope}[x={(image.south east)},y={(image.north west)}]
		        \draw[black, opacity=.5, thick, dashed] (0.5,0.5) circle (0.19);
		    \end{scope}
		\end{tikzpicture}
		\subcaption{$\varphi^h_4$}
                \label{fig:mode4}
	\end{subfigure}%
	~ 
	\begin{subfigure}[t]{0.3\textwidth}
		\centering
		\begin{tikzpicture}
		    \node[anchor=south west,inner sep=0] (image) at (0,0) {\includegraphics[width=0.875\textwidth]{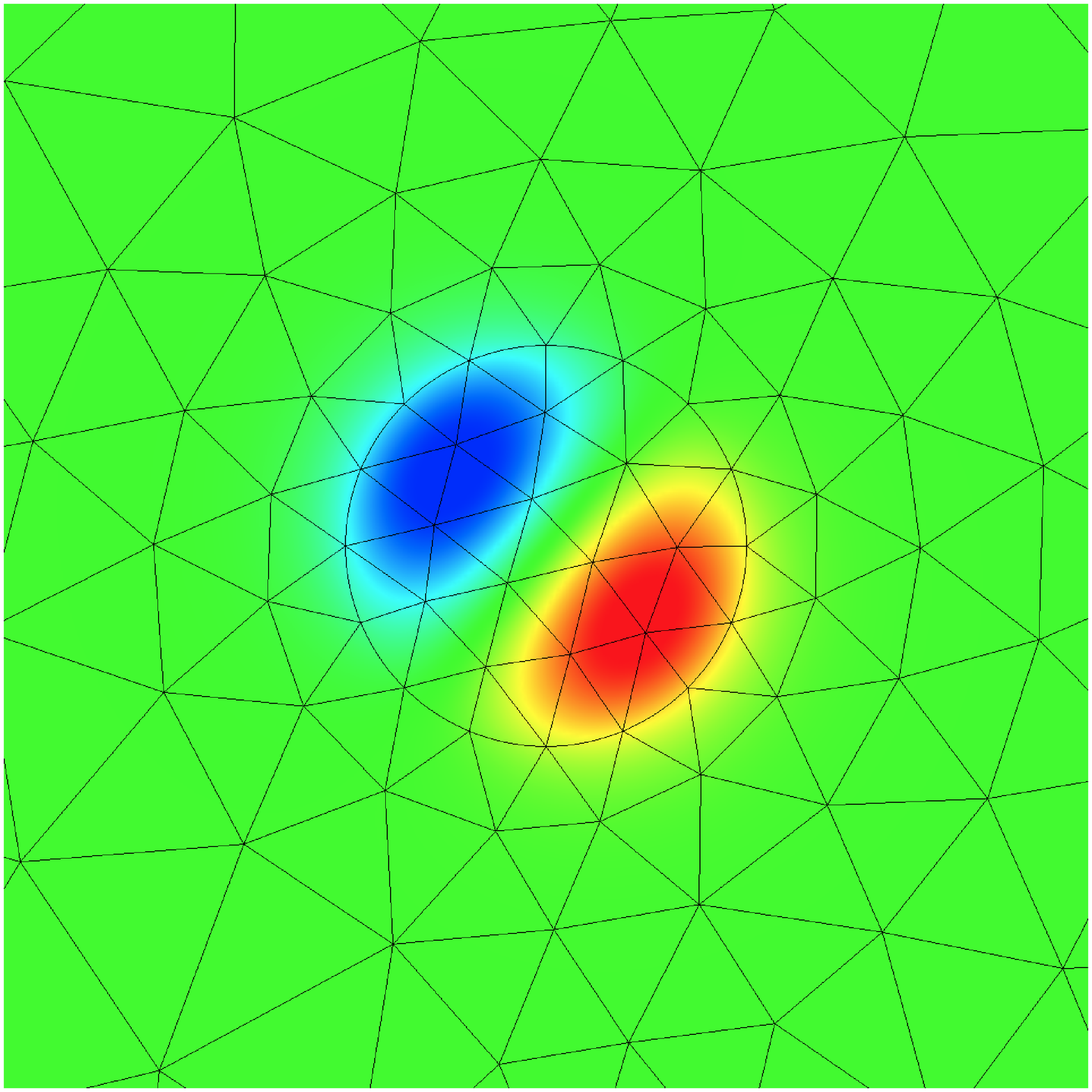}};
		    \begin{scope}[x={(image.south east)},y={(image.north west)}]
		        \draw[black, opacity=.5, thick, dashed] (0.5,0.5) circle (0.19);
		    \end{scope}
		\end{tikzpicture}
		\subcaption{$\varphi^h_5$}
                \label{fig:mode5}
	\end{subfigure}%
	~ 
	\begin{subfigure}[t]{0.3\textwidth}
		\centering
		\begin{tikzpicture}
		    \node[anchor=south west,inner sep=0] (image) at (0,0) {\includegraphics[width=0.875\textwidth]{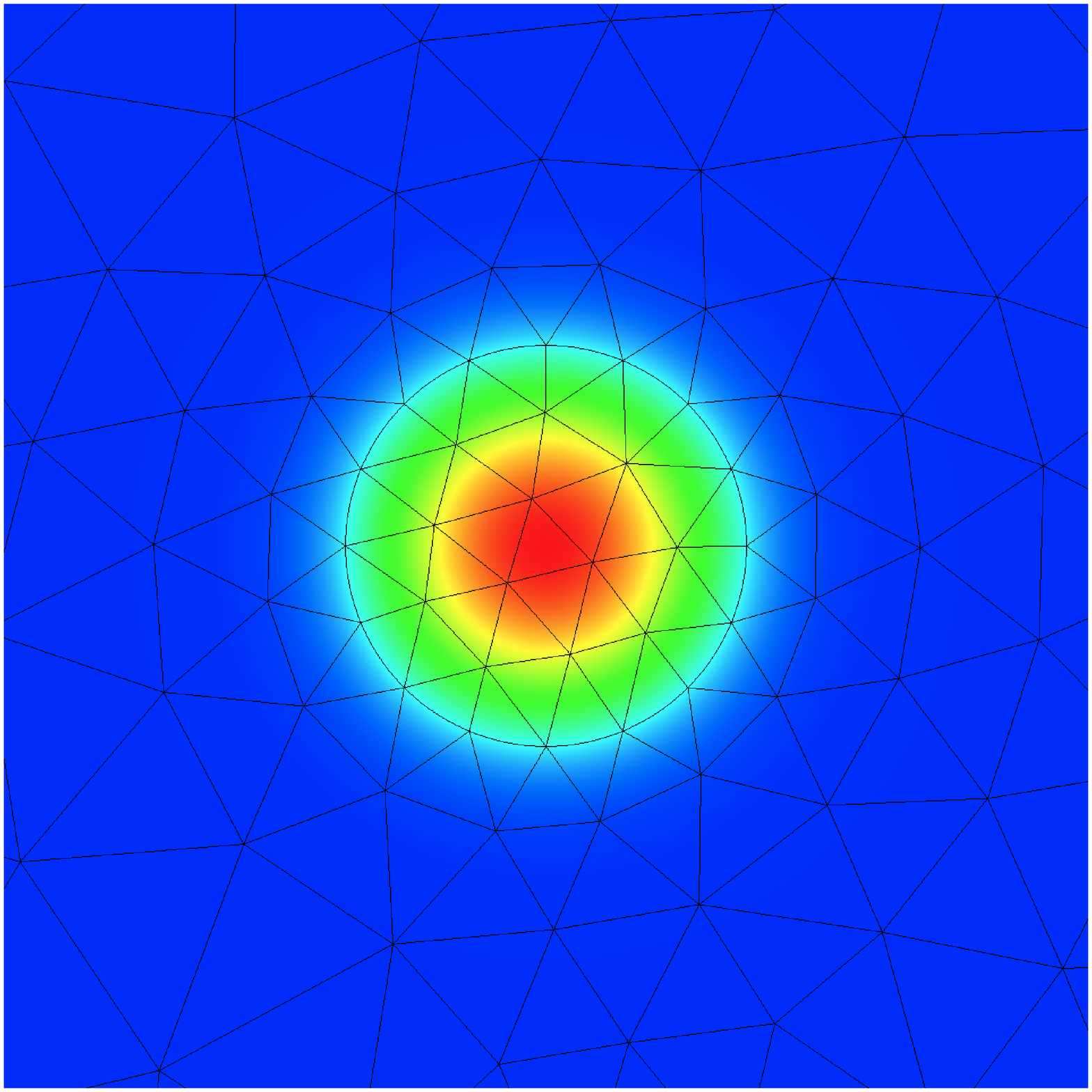}};
		    \begin{scope}[x={(image.south east)},y={(image.north west)}]
		        \draw[black, opacity=.5, thick, dashed] (0.5,0.5) circle (0.19);
		    \end{scope}
		\end{tikzpicture}
		\subcaption{$\varphi^h_6$}
                \label{fig:mode6}
	\end{subfigure}
	\caption{A close view of the approximate eigenfunctions $\varphi_j^h$
          computed by FEAST for the ytterbium-doped fiber. The
          boundary of the fiber core region
          is marked by dashed black circles.}
	\label{fig:ybeigenfuncs}
\end{figure}

We also conducted a convergence study. We began with a mesh whose
approximate mesh size in the core region is $h_c= 1/16$. We performed
three uniform mesh refinements, where each refinement halved the mesh
size. After each refinement, the elements intersecting the core or
cladding boundary were curved again using the geometry information.
Using the DPG discretization and $N = 16$ quadrature points for the
contour integral, we computed the 6 eigenvalues, denoted by
$\hat{\lambda}_l^h$, and compared them with the exact eigenvalues on
the scaled domain, denoted by $\hat{\lambda}_l =
\rcore^2\beta_l^2$. For the parameter values set in~\eqref{eq:Nufern}, there are
six such $\hat{\lambda}_l$ (counting multiplicities) whose approximate values are
$
 \hat{\lambda}_1= 2932065.0334243, \;
  \hat{\lambda}_2 = \hat{\lambda}_3 = 2932475.1036310,\;
  \hat{\lambda}_4= \hat{\lambda}_5=2934248.1978369, \;
  \hat{\lambda}_6=2935689.8561775.
$
Fixing $p=3$, we
report the relative eigenvalue errors
\[
  e_l =\frac{ | \hat{\lambda}_l - \hat{\lambda}_l^h| }{ \hat{\lambda}_l^h}
\]
in Table~\ref{tab:fiber_yb_rates} for each $l$ (columns) and each
refinement level (rows).  A column next to an $e_l$-column indicates
the numerical order of convergence (computed as described in
Section~\ref{NS}\rev{)}.  The observed convergence rates are somewhat near
the order of 6 expected from the previous theory. The match in the
rates is not as close as in the results from the ``textbook''
benchmark examples of Section~\ref{NS}, presumably because mesh
curving may have an influence on the pre-asymptotic behavior. Since
the relative error values have quickly approached machine precision,
further refinements were not performed.

\begin{table}
  \setlength\tabcolsep{4pt}
  \centering
  \begin{footnotesize}    
  \begin{tabular}{|c|cc|cc|cc|cc|cc|cc|}
    \hline 
    core $h$     & $e_1$ & {\tiny{NOC}}     & $e_2$    & {\tiny{NOC}} & $e_3$    &  {\tiny{NOC}}
    & $e_4$   & {\tiny{NOC}} & $e_5$ & {\tiny{NOC}}  & $e_6$ & {\tiny{NOC}}
    \\ \hline 
    $h_c $  & 1.26e-07 &  --  & 2.01e-07 & --  & 1.81e-07 &  -- & 4.99e-08 & --&4.37e-08&-- &1.72e-08 &-- \\ 
    $h_c/2$ & 9.42e-09 & 3.7  & 1.63e-08 & 3.6 & 1.32e-08 & 3.8 & 6.46e-09 &3.0&4.84e-09&3.2&3.38e-09 &2.4 \\  
    $h_c/4$ & 1.17e-10 & 6.3  & 2.13e-10 & 6.3 & 1.80e-10 & 6.2 & 7.03e-11 &6.5&4.84e-11&6.6&3.64e-11 &6.5 \\
    $h_c/8$ & 9.16e-14 & 10.3 & 1.33e-12 & 7.3 & 3.06e-13 & 9.2 & 3.75e-13 &7.6&6.87e-13&6.1&6.69e-14 &9.1 \\
    \hline
  \end{tabular}
  \end{footnotesize}
  \caption{Convergence rates of the fiber eigenvalues}
  \label{tab:fiber_yb_rates}
\end{table}

\input{references.tex}
\end{document}

%% file: definition.tex
\usepackage{amssymb,amsmath,amsthm}

\usepackage[margin=2.5cm,footskip=1.5cm,headsep=1.5cm]{geometry}
\usepackage{tikz,pgfplots}
\usepackage{algorithm}
\usepackage{algorithmic}
\usepackage{subcaption}
\usepackage{cite}
\usepackage{listings}
\usepackage{multirow}
\usepackage{color}
\usepackage[centertags]{amsmath}
\usepackage{mathabx}
\usepackage{enumitem}
\setlist[itemize]{noitemsep, topsep=0pt, leftmargin=0cm}
\usepackage{amscd}

\usepackage{color}

\usepackage{epstopdf}
\usepackage{alltt}
\usepackage{float}
\usepackage{url}

\usepackage{framed}
\newtheorem{theorem}{Theorem}[section]
\newtheorem{lemma}[theorem]{Lemma}

\theoremstyle{remark}

\newtheorem{assumption}[]{Assumption}
\newtheorem{remark}[theorem]{Remark}

 \DeclareMathOperator{\grad}{grad}

\DeclareMathOperator{\gap}{gap}

\newcommand{\Creg}{C_{\mathrm{reg}}}

\newcommand{\om}{\varOmega}
\newcommand{\oh}{\varOmega_h}

\newcommand{\Poincare}{Poincar{\'{e}}}     
\newcommand{\ip}[1]{\langle {#1} \rangle}
\newcommand{\veps}{\varepsilon}
\newcommand{\dist}[1]{\mathop{\textstyle{\mathrm{dist}_{{#1}}}}}
\newcommand{\dom}{\mathop{\mathrm{dom}}}
\newcommand{\diam}{\mathop{\mathrm{diam}}}

\newcommand{\CCC}{\mathbb{C}}
\renewcommand{\div}{\mathrm{div}}

\newcommand{\RT}{RT}

\newcommand{\R}{\mathbb{R}}
\newcommand{\C}{\mathbb{C}}

\newcommand{\cA}{A}

\newcommand{\cH}{{\mathcal H}}

\newcommand{\cV}{{\mathcal V}}

\newcommand{\RR}{\mathbb{R}}

\newcommand{\ii}{\mathrm{i}}

\newcommand{\logLogSlopeTriangle}[5]
{

    \pgfplotsextra
    {
        \pgfkeysgetvalue{/pgfplots/xmin}{\xmin}
        \pgfkeysgetvalue{/pgfplots/xmax}{\xmax}
        \pgfkeysgetvalue{/pgfplots/ymin}{\ymin}
        \pgfkeysgetvalue{/pgfplots/ymax}{\ymax}

        \pgfmathsetmacro{\xArel}{#1}
        \pgfmathsetmacro{\yArel}{#3}
        \pgfmathsetmacro{\xBrel}{#1-#2}
        \pgfmathsetmacro{\yBrel}{\yArel}
        \pgfmathsetmacro{\xCrel}{\xArel}

        \pgfmathsetmacro{\lnxB}{\xmin*(1-(#1-#2))+\xmax*(#1-#2)} 
        \pgfmathsetmacro{\lnxA}{\xmin*(1-#1)+\xmax*#1} 
        \pgfmathsetmacro{\lnyA}{\ymin*(1-#3)+\ymax*#3} 
        \pgfmathsetmacro{\lnyC}{\lnyA+#4*(\lnxA-\lnxB)}
        \pgfmathsetmacro{\yCrel}{\lnyC-\ymin)/(\ymax-\ymin)} 

        \coordinate (A) at (rel axis cs:\xArel,\yArel);
        \coordinate (B) at (rel axis cs:\xBrel,\yBrel);
        \coordinate (C) at (rel axis cs:\xCrel,\yCrel);

        \draw[#5]   (A)-- node[pos=0.5,anchor=north] {{\tiny{$1$}}}
                    (B)-- 
                    (C)-- node[pos=0.5,anchor=west] {{\tiny{#4}}}
                    cycle;
    }
}

\newcommand{\RRR}{\mathbb{R}}

\newcommand{\Eh}[1]{{E_h^{({#1})}}}

\newcommand{\hdist}{\mathop{\mathrm{dist}}}

\newcommand{\lmax}{\Lambda_h^{\mathrm{max}}}

\renewcommand{\Omega}{\om}

\newcommand{\rcore}{r_{\text{core}}}
\newcommand{\rclad}{r_{\text{clad}}}
\newcommand{\nclad}{n_{\text{clad}}}
\newcommand{\ncore}{n_{\text{core}}}

 \newcommand{\rev}[1]{{#1}}
 \newcommand{\revv}[1]{{#1}}
